\documentclass[11pt,a4paper]{amsart}

 \usepackage{graphics, datetime, url, subfig, graphicx}

\usepackage{amssymb}
\usepackage{amsmath,graphics,verbatim}
\usepackage{latexsym}
\usepackage{eucal}
\usepackage{a4wide}

\newtheorem{theorem}{Theorem}[section]
\newtheorem{proposition}[theorem]{Proposition}
\newtheorem{lemma}[theorem]{Lemma}

\newtheorem{corollary}[theorem]{Corollary}
\theoremstyle{remark}
\newtheorem{remark}[theorem]{Remark}

\newtheorem{definition}[theorem]{Definition}

\newcommand\RR{\mathbb{R}}

\def\dist{d}

\renewcommand\Re{\operatorname{Re}}

\newcommand\CI{C^\infty}

\def\jap#1{\langle #1 \rangle}

\newcommand\Id{\operatorname{Id}}
\newcommand{\zf}{\mathrm{zf}}

\newcommand{\lb}{\mathrm{lb}}
\newcommand{\rb}{\mathrm{rb}}

\newcommand{\bfc}{\mathrm{bf}}

\newcommand{\ff}{\mathrm{ff}}
\renewcommand{\sf}{\mathrm{sf}}

\newcommand\norsc{N_{sc}}
\newcommand\ib{I_b}

\newcommand{\lbz}{\operatorname{lbz}}
\newcommand{\rbz}{\operatorname{rbz}}
\newcommand{\lbi}{\operatorname{lbi}}
\newcommand{\rbi}{\operatorname{rbi}}

\begin{document}
\title[Riesz transform on metric cones]{The Riesz transform for homogeneous Schr\"odinger operators on metric cones}

\author{Andrew Hassell}
\address{Department of Mathematics, Australian National University \\ Canberra ACT 0200 \\ AUSTRALIA}
\email{Andrew.Hassell@anu.edu.au}

\author{Peijie Lin}
\address{Department of Mathematics, Australian National University \\ Canberra ACT 0200 \\ AUSTRALIA}
\email{Paul.Lin@anu.edu.au}

\thanks{The first author was supported by an ARC Future Fellowship FT0990895 and an ARC Discovery Grant DP1095448}

\begin{abstract}
We consider Schr\"odinger operators on metric cones whose cross section is  a closed Riemannian manifold $(Y, h)$ of dimension $d-1 \geq 2$. Thus the metric on the cone $M = (0, \infty)_r \times Y$ is $dr^2 + r^2 h$. Let $\Delta$ be the Friedrichs Laplacian on $M$ and $V_0$ be a smooth function on $Y$, such that $\Delta_Y + V_0 + (d-2)^2/4$ is a strictly positive operator on $L^2(Y)$, with lowest eigenvalue $\mu^2_0 $ and second lowest eigenvalue $\mu^2_1$, with $\mu_0, \mu_1 > 0$. The operator we consider is $H = \Delta + V_0/r^2$, a Schr\"odinger operator with inverse square potential on $M$; notice that $H$ is homogeneous of degree $-2$. 

We study the Riesz transform
$
T = \nabla H^{-1/2}
$
and determine the precise range of $p$ for which $T$ is bounded on $L^p(M)$. This is achieved by making a precise analysis of the operator $(H + 1)^{-1}$ and determining the complete asymptotics of its integral kernel. We prove that if $V$ is not identically zero, then the range of $p$ for $L^p$ boundedness is 
$$
\Bigg(\frac{d}{\min(1+\frac{d}{2}+\mu_0, d)} \, , \,  \frac{d}{\max(\frac{d}{2}-\mu_0, 0)}\Bigg),
$$
while if $V$ is identically zero, then the range is 
$$
\Bigg(1 \, , \,  \frac{d}{\max(\frac{d}{2}-\mu_1, 0)}\Bigg).
$$
The result in the case $V$ identically zero was first obtained  in a paper by H.-Q. Li \cite{HQL}. 
\end{abstract}

\maketitle

\section{Introduction}
\label{chapter1}


The Riesz transform $T$ on the Euclidean space $\mathbb{R}^d$ is defined by 
\begin{equation*}
T=\nabla \Delta_{\mathbb{R}^d}^{-\frac{1}{2}},
\end{equation*}
where $\Delta_{\mathbb{R}^d}$ is the Laplacian operator. 
In this paper we study the Riesz transform $T$ in a more general setting of metric cones. 
A metric cone $M$ is of the form $M = Y\times (0,\infty)$, where $(Y, h)$ is a compact Riemannian manifold with dimension $d-1$. 
The cone $M$ is equipped with the conic metric $g=dr^2+r^2h$. 
  The Euclidean space $\mathbb{R}^d$ provides the simplest example of a metric cone, 
with cross section $Y=\mathbb{S}^{d-1}$ with its standard metric. General metric cones enjoy a dilation symmetry analogous to that of Euclidean space, but no other symmetries in general. 


The Laplacian on the cone expressed in polar coordinates is 
\begin{equation}\label{laplaciandefinition}
\Delta=-\partial^2_r-\frac{d-1}{r}\partial_r+\frac{1}{r^2}\Delta_Y,
\end{equation}
where $\Delta_Y$ is the Laplacian on the compact Riemannian manifold $Y$.
Then the Riesz transform $T$ on the cone $M$ is defined by
$T=\nabla \Delta^{-\frac{1}{2}}$, 
where $\nabla$ is shorthand for $(\partial_r, r^{-1} \nabla_Y)$, or in other words we measure the gradient on the cone using the metric $g$. 
The question of the boundedness of the Riesz transform on cones, i.e.\ for what $p$ the operator $T$ is bounded on $L^p(M)$, 
was answered by H.-Q. Li in \cite{HQL}. 
The characterisation of the boundedness, stated in Theorem \ref{HQLR}, 
is in terms of the second smallest eigenvalue of an operator involving $\Delta_Y$. 
We provide a different proof to this result in Section~\ref{chapter6} of this paper. 

\begin{theorem}\label{HQLR}
Let $d\geq 3$, and $M$ be a metric cone with dimension $d$ and cross section $Y$. 
The Riesz transform $T=\nabla\Delta^{-\frac{1}{2}}$ is bounded on $L^p(M)$ if and only if 
$p$ is in the interval
\begin{equation}\label{cinterval0}
\Bigg(1, \ \frac{d}{\max(\frac{d}{2}-\mu_1, 0)}\Bigg),
\end{equation}
where $\mu_1>0$ is the square root of the second smallest eigenvalue of the operator $\Delta_Y+(\frac{d-2}{2})^2$.
\end{theorem}

More significantly, the methods used in this paper to prove Theorem \ref{HQLR}
can be applied to study the boundedness properties of a more generalised class of operators, 
obtained by introducing an inverse square potential to the Laplacian. 
Let $V_0: Y\rightarrow\mathbb{C}$ be a function on $Y$ satisfying the condition
\begin{equation}
\Delta_Y+V_0(y)+(\frac{d-2}{2})^2>0
\label{V0condition}\end{equation}
and define 
\begin{equation}
H = \Delta + \frac{V_0}{r^2}.
\label{Hdefn}\end{equation}
Notice that $H$ is homogeneous of degree $-2$, like the Laplacian. Condition \eqref{V0condition} ensures that  $H$ is a strictly positive operator, so $H^{-1/2}$ is well-defined. We can then define 
the Riesz transform $T$ of the Schr\"odinger operator $H$ by
\begin{equation}
T= \nabla H^{-1/2} = \nabla\bigg(\Delta+\frac{V_0(y)}{r^2}\bigg)^{-\frac{1}{2}}.
\label{Tpotdefn}\end{equation}
Notice that \eqref{V0condition} allows our potential $V=\frac{V_0}{r^2}$ to be ``a bit negative''; in particular, it allows $V_0$ to be any constant greater than $-(d-2)^2/4$. 

The goal of this article is to find the exact interval for $p$ on which 
the Riesz transform $T$ with an inverse square potential $V={V_0}/{r^2}$ is
bounded on $L^p(M)$, where $M$ is a metric cone with dimension $d\geq 3$. 

A necessary condition, stated in Theorem \ref{Andrew}, for the boundedness was found
in \cite{GH} by C. Guillarmou and the first author, in a slightly different setting --- asymptotically conic manifolds. 
These are complete Riemannian manifolds $(M^\circ, g)$ such that $M^\circ$ is the interior of a compact manifold with boundary, $\overline{M}$, which has a boundary defining function $x$ for which the metric $g$ has the form 
$$\frac{dx^2}{x^4} + \frac{h(x)}{x^2},$$ in a collar neighbourhood of $\partial \overline{M}$, where $h(x)$ is a family of metrics on $\partial \overline{M}$. Here $r = 1/x$ behaves like the radial coordinate on the cone over $\partial \overline{M}$: the metric in terms of $r$ reads $g = dr^2 + r^2 h(1/r)$, so is asymptotic to the conic metric $dr^2 + r^2 h(0)$ as $r \to \infty$. In \cite{GH}, potentials of the form $V \in x^2 \CI(\overline{M})$ were considered; that is, the potentials decay as $r^{-2}$ at infinity, and the  limiting `potential at infinity' $V_0$ was defined by $V_0 := x^{-2} V |_{\partial \overline{M}}$. 

\begin{theorem}\label{Andrew}(\cite[Theorem 1.5]{GH})
Let $d\geq 3$, and $(M^\circ, g)$ be an asymptotically conic manifold with dimension $d$. 
Consider the operator $\mathbf{P}=\Delta_g+V$ with $V \in x^2 \CI(\overline{M})$ satisfying 
\begin{equation}\label{posop}
\Delta_{\partial M}+V_0+\bigg(\frac{d-2}{2}\bigg)^2>0\mbox{ on $L^2(Y)$, where }V_0=\frac{V}{x^2}\bigg|_{\partial M}.
\end{equation}
Let $\mu_0 > 0$ be the square root of the lowest eigenvalue of the operator \eqref{posop}. 
Suppose that $\mathbf{P}$ has no zero modes or zero resonance
and that $V_0\not\equiv 0$. Then $\nabla \mathbf{P}^{-1/2}$
is unbounded on $L^p(M)$ if $p$ is outside the interval 
\begin{equation}\label{Aninterval}
\Bigg(\frac{d}{\min(\frac{d}{2}+1+\mu_0, d)}, \ \frac{d}{\max(\frac{d}{2}-\mu_0, 0)}\Bigg). 
\end{equation}
\end{theorem}

The counter-example used in \cite{GH} to show the unboundedness of the Riesz transform 
can be easily adapted to the context of 
metric cones, so a similar result also holds for metric cones. 
Therefore the task now is to find a sufficient condition for  boundedness. 
We will see 
that the sufficient condition involves the same interval (\ref{Aninterval}) as in Theorem \ref{Andrew},  
so this interval gives us a complete characterisation of the boundedness of $T$ with $V\not\equiv 0$. 
Our main result  is as follows.

\begin{theorem}\label{fmain}
Let $d\geq 3$, and $M$ be a metric cone with dimension $d$ and cross section $Y$. 
Let $V_0$ be a function on $Y$ that satisfies $\Delta_Y+V_0(y)+(\frac{d-2}{2})^2>0$.
The Riesz transform $T$ with the inverse square potential $V=\frac{V_0}{r^2}$ is bounded on $L^p(M)$ for $p$ in the interval
\begin{equation}\label{ccinterval}
\Bigg(\frac{d}{\min(1+\frac{d}{2}+\mu_0, d)}, \ \frac{d}{\max(\frac{d}{2}-\mu_0, 0)}\Bigg),
\end{equation}
where $\mu_0>0$ is the square root of the smallest eigenvalue of the operator $\Delta_Y+V_0(y)+(\frac{d-2}{2})^2$.

Moreover, for any $V\not\equiv 0$, the interval (\ref{ccinterval}) characterises the boundedness of $T$, 
ie $T$ is bounded on $L^p(M)$ if and only if $p$ is in the interval (\ref{ccinterval}).
\end{theorem}

\begin{remark} If we specialize to \emph{positive} potentials, i.e.\  $V\geq 0$ and $V\not\equiv 0$, then $\mu_0 > (d-2)/2$, and we see that 
the lower threshold for $L^p$ boundedness is $1$, 
and the upper threshold is always greater than $d$. 
On the other hand, for \emph{negative} potentials $V$, i.e.\  $V\leq 0$ and $V\not\equiv 0$, 
the lower threshold for the $L^p$ boundedness is always greater than $1$ and strictly less than $2$, while the upper threshold is strictly less than $d$ but strictly larger than $2$. 
\end{remark}

\begin{remark} Part of these results are implied by a recent paper of Assaad \cite{JA} dealing with more general classes of potentials on $\RR^d$ or on complete Riemannian manifolds; see the end of Section~\ref{literaturereview} for further discussion.
\end{remark}

An immediate application of Theorem \ref{fmain} is to show that 
the converse of the second part of \cite[Theorem 1.5]{GH}, i.e.\  Theorem~\ref{Andrew}, is also true. 
As noted in \cite[Remark 1.7]{GH}, Theorem \ref{fmain} is exactly the missing ingredient. 
Therefore we have the following result. 

\begin{theorem} Let $(M^\circ, g)$,  $\mathbf{P}$ and $\mu_0$ be as in Theorem~\ref{Andrew}. Then the 
Riesz transform $\nabla P^{-1/2}$ is bounded on $L^p(M)$ \textbf{if and only if} $p$ is in the interval \eqref{Aninterval}. 
\end{theorem}

A special case of Theorem \ref{fmain} is the following result on the Riesz transforms 
with constant non-zero $V_0$, in which the boundedness interval is written in terms of the constant. 

\begin{corollary}\label{vequalcintro}
Let $M$ be a metric cone with dimension $d \geq 3$ and cross section $Y$. 
The Riesz transform $T=\nabla(\Delta+\frac{c}{r^2})^{-\frac{1}{2}}$, 
where $c>-(\frac{d-2}{2})^2$ and $c\ne 0$, is bounded on $L^p(M)$ 
if and only if $p$ is in the interval
\begin{equation}\label{vequalcinin}
\Bigg(\frac{2d}{\min(d+2+\sqrt{(d-2)^2+4c}, 2d)}, \ \frac{2d}{\max(d-\sqrt{(d-2)^2+4c}, 0)}\Bigg).
\end{equation}
\end{corollary}

\subsection{Strategy of the proof}

Using functional calculus, we get the following expression,
\begin{equation}
T=\frac{2}{\pi}\int_0^\infty \nabla( H+\lambda^2)^{-1}d\lambda.
\label{fc}\end{equation}
Because of the homogeneity of $H$, we obtain the resolvent kernel for $(H + \lambda^2)^{-1}$ from $(H + 1)^{-1}$ by scaling the variables. So it suffices to analyze $P := (H+1)^{-1}$. We do this on a compactified and blown up space, which is designed so that the asymptotics of its kernel in different regimes can be understood. Let us use $y$ as a local coordinate on the cross section $Y$. We particularly want to  distinguish the diagonal behaviour of the kernel $P^{-1}(r, y, r', y')$, from the behaviour as $r$ or $r'$ tend to zero or infinity. If we consider the kernel as living on $(Y \times [0, \infty])^2$,  as in Figure~\ref{domain}, then this has the defect that the diagonal meets the boundary hypersurfaces $\{ r = 0 \}$, $\{ r' = 0 \}$, $\{ r = \infty \}$ and  $\{ r' = \infty \}$, making the different asymptotic behaviours difficult to distinguish. To remedy this we perform blowups, as in \cite{GH}. As noted in that paper, the operator $r P r$ is elliptic as a $b$-differential operator near $r = 0$, that is, an elliptic combination of the `$b$-vector fields' $r \partial_r$ and $\partial_{y_i}$. On the other hand,  as $r \to \infty$, $P$ is an elliptic \emph{scattering} differential operator, which is to say that it has an expression that looks like the Euclidean Laplacian in polar coordinates as $r \to \infty$, being an elliptic combination of $\partial_r$ and $r^{-1} \partial_{y_i}$. Correspondingly we perform the $b$-blowup (used to define the $b$-calculus --- see Section~\ref{sec:bsc}) for small $r$, that is, blow up the corner $r = r' = 0$, while for large $r$ we perform two blowups (used to define the scattering calculus), namely we first blow up the corner $r = r' = \infty$, followed by the boundary of the lifted diagonal at $r = \infty$, obtaining the space illustrated in Figure~\ref{domain}. Now the diagonal is separated from the boundary hypersurfaces in Figure~\ref{domain} and on this blown-up space,  we can more easily construct the kernel of $P^{-1}$ and describe the different types of asymptotics.

\begin{figure}[h]
\centering
\includegraphics[scale=0.35]{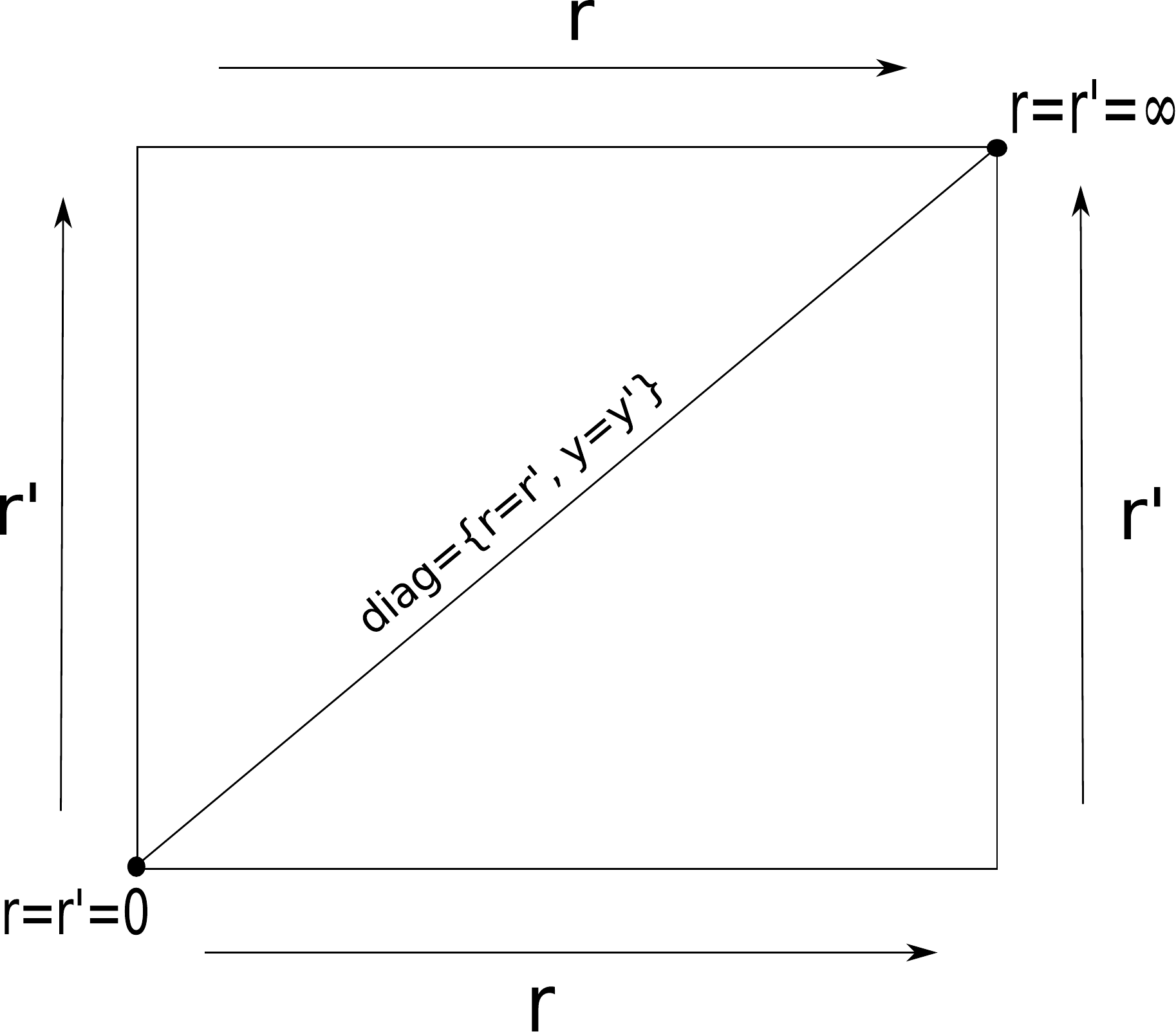} \qquad  \qquad 
\includegraphics[scale=0.30]{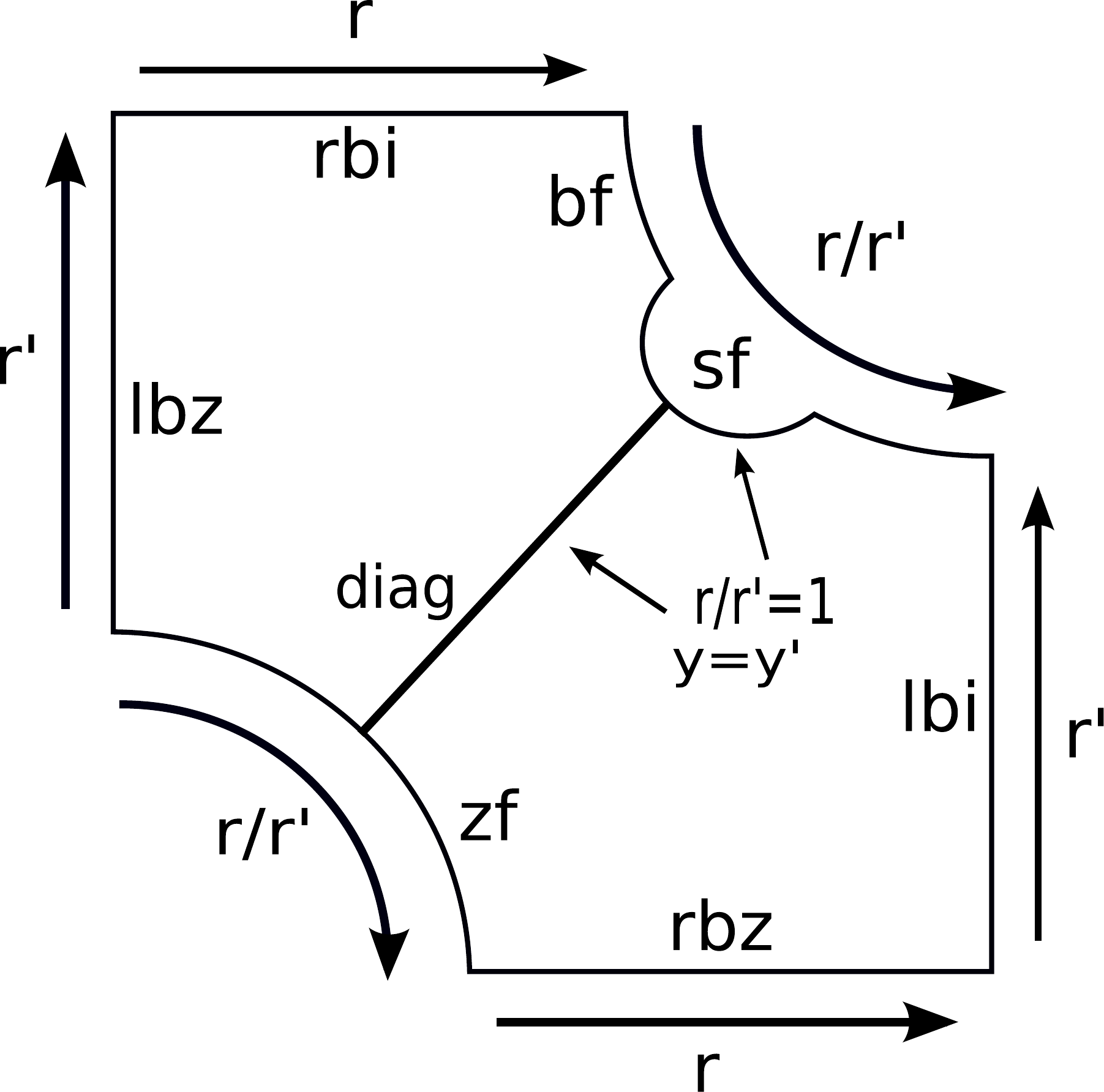}
\caption{The space $(Y \times [0, \infty])^2$, left, and the blown-up double space after three blowups, right.}
\label{domain}
\end{figure}

Because the kernel behaves differently in different parts of the blown-up space, and 
especially because we use different calculi near the two hypersurfaces $\zf$ and $\sf$, 
we break the blown-up space into different regions, and construct 
the resolvent kernel in each region separately using different tools and techniques.
In the end we patch up the constructions in these different regions to obtain the overall resolvent kernel. 
This construction of the resolvent kernel of $H$, ie the kernel of $P^{-1}$, is done in Section~\ref{chapter5}. 

In Section~\ref{chapter6}, equipped with the knowledge on the behaviours of the kernel of $P^{-1}$ 
at different parts of the blown-up space, 
we determine the 
boundedness properties of the Riesz transform $T$. 
Using a smooth partition of unity on the blown up space, we perform the integral \eqref{fc} and then break the kernel of $T$ up into a near-diagonal part and an off-diagonal part. 
The near-diagonal part is a Calder\'on-Zygmund kernel and is bounded on $L^p$ for all $p \in (1, \infty)$, while the off-diagonal part is bounded on typically a smaller range of $p$ determined by the leading asymptotic behaviour at the boundary hypersurfaces marked `$\lbz$' or `$\rbz$' in Figure~\ref{domain}.

\subsection{Relation to previous work}\label{literaturereview}

Cones have been studied since the 19th century, particularly the problem of wave diffraction
from a cone point which is important in applied mathematics, 
for example in \cite{ASO} by A. Sommerfeld. 
Other notable early papers include \cite{FGF} and \cite{FGF2} by F. G. Friedlander and \cite{BKeller} by A. Blank and J. B. Keller.
The Laplacians defined on cones were studied by J. Cheeger and M. Taylor in \cite{CT} and \cite{CT2}. Many papers have been written about spaces with cone-like singularities. For example, 
the Laplacian and heat kernel on compact Riemannian manifolds with cone-like singularities 
has been studied in \cite{JC} by J. Cheeger and in \cite{EM} by E. Mooers,  in \cite{BS}, J. Br\"uning and R. Seeley studied the Laplcian on manifolds with an asymptotically conic singularity, and in \cite{MW} R. B. Melrose and J. Wunsch study the wave equation and diffraction on spaces with asymptotically conic singularities. 

The classical case of the Riesz transform on the Euclidean space $\mathbb{R}^d$ goes back to the 1920s, 
and the case of one dimension (the Hilbert transform) was studied by M. Riesz in \cite{MR}.
The paper \cite{RSS} by R. S. Strichartz is the first paper that studies the Riesz transform 
on a complete Riemannian manifold. 
In \cite{CD} T. Coulhon and X. T. Duong proved 
that the Riesz transform on a complete Riemannian manifold, 
satisfying the doubling condition and the diagonal bound on the heat kernel,
is of weak type $(1,1)$, and hence is bounded on $L^p$ for $1<p\leq 2$. Since then, there have been many studies of the Riesz transform, of which we mention just a few: studies of the Riesz transform on complete Riemannian manifolds include \cite{CD03}, \cite{ACDH}, \cite{AC}, \cite{CND}; on Lie groups include \cite{GA92}, \cite{GA02}, \cite{ERS}, \cite{LSC}; on second order elliptic operators \cite{BK}, \cite{ERS01}. 

Many papers have been written on Schr\"odinger operators with an inverse square potential. 
We only mention a few of the most relevant ones here. 
In \cite{XPW}, X. P. Wang studied the perturbations of such operators.
In \cite{GC}, G. Carron studied Schr\"odinger operators with potentials that are homogeneous of degree $-2$ near infinity.
In \cite{BPSTZ} by N. Burq, F. Planchon, J. G. Stalker and A. S. Tahvildar-Zadeh, 
the authors generalise the corresponding standard Strichartz estimates 
of the Schr\"odinger equation and the wave equation to the case
 in which an additional inverse square potential is present.
In \cite{HS} the first author and A. Sikora investigated one-dimensional Riesz transforms, including with inverse square potentials, with respect to measures of the form $r^{d-1} dr$, thus mimicking the measure on a $d$-dimensional cone. 

Now we turn to past results on 
the boundedness of the Riesz transform $T$ with a potential $V$ on metric cones. We have already mentioned the result (Theorem \ref{HQLR}) of H.-Q. Li for $V \equiv 0$, and the work \cite{GH} of 
C. Guillarmou and the first author on asymptotically conic manifolds. The method from \cite{GH} was based in part on the paper \cite{CCH}.  In \cite{GH2} the two authors performed a similar analysis  
but allowed zero modes and zero resonances. 
In \cite{ABA}, P. Auscher and B. Ben Ali obtained a result on $\mathbb{R}^d$, stated in Theorem \ref{ABAR}, which involves the reverse H\"older condition.
It is an improvement of the earlier results by Z.W. Shen in \cite{ZWS}. 

\begin{theorem}(\cite[Theorem 1.1]{ABA})\label{ABAR}
Let $1< q\leq \infty$. If $V$ is in the reverse H\"older class $B_q$ then for some $\varepsilon>0$ depending only on $V$
the Riesz transform with potential $V$ is bounded on $L^p(\mathbb{R}^d)$ for $1< p< q+\varepsilon$. 
\end{theorem}

The reverse H\"older condition $V\in B_q$, implies that $V > 0$ almost everywhere and $V\in L_{loc}^q(\mathbb{R}^d)$. A positive inverse square potential is in $B_q$ if and only if $q < d/2$. So this theorem gives boundedness for $1 < p < d/2$, which is smaller than the range obtained in Theorem~\ref{fmain} for positive inverse square potentials (of course this is a very small subclass of $B_q$-potentials). 

Very recently, Assaad and Assaad-Ouhabaz have proved results for Riesz transforms of Schr\"odinger operators which include some of our results. The following result is from \cite{JA}:

\begin{theorem}\label{assaadresult}
Let $M$ be a non-compact complete Riemannian manifold with dimension $d\geq 3$. 
Suppose that the function $V\leq 0$ satisfies $\Delta+(1+\varepsilon)V\geq 0$, 
the Sobolev inequality 
\begin{equation}\label{Sobineq}
||f||_{L^{\frac{2d}{d-2}}(M)}\lesssim ||\nabla f||_{L^2(M)},
\end{equation}
holds for all $f\in C_0^\infty(M)$, 
and that $M$ is of homogeneous type, ie for all $x\in M$ and $r>0$,
\[\mu\big(B(x, 2r)\big)\lesssim\mu\big(B(x, r)\big),\]
where $\mu$ is the measure on $M$. 
Then the Riesz transform $T=\nabla(\Delta+V)^{-\frac{1}{2}}$ is 
bounded on $L^p(M)$ for all $p$ in the interval
\begin{equation}\label{assaadinterval}
\Bigg(\frac{2d}{d+2+(d-2)\sqrt{\frac{\varepsilon}{\varepsilon+1}}}, \hspace{3mm}2\Bigg].
\end{equation}
\end{theorem} 

This result can be directly compared with ours in the case of Schr\"odinger operators of the form $(\Delta+\frac{c}{r^2})$, 
where the constant $c$ satisfies $-(\frac{d-2}{2})^2<c<0$. In that case 
the lower threshold in (\ref{vequalcinin}) given by Corollary  \ref{vequalcintro} 
is the same as the lower threshold in (\ref{assaadinterval}) given by J. Assaad's result.
Also in \cite{JA} it is shown that the Riesz transform for Schr\"odinger operators with potentials in $L^{d/2}$ on a $d$-dimensional Riemannian manifold obeying the Sobolev inequality 
\eqref{Sobineq}
are bounded on $(1,d)$ provided that this is true for the Riesz transform with zero potential. Note that this case just fails to cover inverse square potentials on cones, which are in $L^{d/2, \infty}$. 
Further results on signed potentials are proved by 
 J. Assaad and E. Ouhabaz in \cite{AO}. 


\section{Review of the $b$-calculus and the scattering calculus}
\label{sec:bsc}

In this section we briefly recall the key elements of the $b$-calculus and the scattering calculus that we require in Section~\ref{chapter5}. For more details, see \cite{RM93} or \cite{GD} for the $b$-calculus, and \cite{RM94, RM95}  for the scattering calculus. 

\subsection{$b$-calculus}\label{subsec:b}
 Let $X$ be a manifold with boundary and with boundary defining function $x$ (that is, $\partial X = \{ x = 0 \}$ and $dx \neq 0$ at $\partial X$). 
The $b$-calculus is a ``microlocalization'' of the set of $b$-differential operators, namely those generated over $C^\infty(X)$ by vector fields tangent to the boundary of $X$; near $\partial X$ such vector fields are a linear combination of vector fields $x \partial_x$ and $\partial_{y_i}$, in terms of a local coordinate system $(x, y_1, \dots, y_{d-1})$ with $(y_1, \dots, y_{d-1})$ restricting to a local coordinate system on $\partial X$. 

It is convenient to regard such operators as acting on $b$-half densities, that is, multiples of a half-density taking the form $|dx/x dy_1 \dots dy_{d-1} |^{1/2}$ near the boundary. Correspondingly, the Schwartz kernels of such operators may be written as a distribution tensored with a $b$-half density in each of the left and right variables.

To define the $b$-calculus, we first blow up\footnote{Here and below we use `blow up' to mean real blow up; as a set, the manifold $[X;S]$ obtained by blowing up $X$ at the submanifold $S$ is obtained  by removing $S$ and replacing it with its inward pointing spherical normal bundle. It is endowed with a differentiable structure that makes polar coordinates around $S$ smooth functions on the blown up space.}  $X^2$ along $(\partial X)^2$ 
to obtain the blown-up manifold $X_b^2=[X^2;(\partial X)^2]$, called the $b$-double space. This produces a manifold with corners which has three boundary hypersurfaces: one defined by $x/x' = 0$ (here and below we use the convention that unprimed variables on the double space are coordinates on the left copy, and primed variables are coordinates on the right copy), one defined by $x'/x = 0$ and one defined by $x+x' = 0$. These are usually denoted $\lb, \rb$ and $\ff$ respectively, but in accordance with our notational conventions for the space in Figure~\ref{domain}, we will call them $\lbz$, $\rbz$ and $\zf$ here (the `z' stands for `zero' here and refers to the fact that the $b$-blowup takes place at $r=r' = 0$). 

The \emph{small $b$-calculus} $\Psi_b^m(X)$, $m\in\mathbb{R}$, is defined as the set of $b$-half-density-valued 
distributions $u$ on $X_b^2$ satisfying 
\begin{enumerate}
\item[\rm (i)] $u$ is conormal of order $m$ with respect to diag$_b$, smoothly up to the
hypersurface $\zf$; 
\item[\rm (ii)] $u$ vanishes to infinite order at $\lbz$ and $\rbz$.
\end{enumerate}
Using the Schwartz kernel theorem, we interpret these as operators on (smooth functions) on $X$; the space $\Psi_b^0(X)$ extends to a bounded operator on $L^2$. 
We also define
\[\Psi_b^{-\infty}(X)=\bigcap_m\Psi_b^m(X);\]
such operators are simply smooth 
$b$-half-densities that vanish at $\lbz$ and $\rbz$.

The $b$-calculus is  closed under composition; see  \cite[Prop 5.20]{RM93} for the proof of the following result.

\begin{proposition}
If $X$ is a compact manifold with boundary then
\[\Psi_b^m(X)\circ\Psi_b^{m'}(X)\subset\Psi_b^{m+m'}(X),\]
where $m, m'\in\mathbb{R}$.
\end{proposition}

Since our purpose is to invert elliptic $b$-differential operators, it's important to know about parametrix constructions 
under the small $b$-calculus. It is analogous to \cite[Theorem 18.1.24]{LH3}.

\begin{proposition}\label{invert1}
If $P$ is an elliptic partial differential operator of order $k$, 
then there exists an operator $G$ in the small $b$-calculus of order $-k$ such that 
\[\Id-PG\in\Psi_b^{-\infty}(X), \quad \Id-GP\in\Psi_b^{-\infty}(X),\]
and $G$ with this property is unique up to an element of $\Psi_b^{-\infty}$.
\end{proposition}

For the proof, see \cite[Sec. 4.13]{RM93}. This inversion property is not good enough for  Fredholm theory, as the error terms $\Id-PG, \Id-GP$ may not be  compact. 
To investigate when an element in the small $b$-calculus is compact, we introduce the indicial operator.

\begin{definition}\label{def:indicial}
Let $A \in \Psi_b^m(X)$ be a $b$-pseudodifferential operator. The indicial operator $\ib{}(A)$ is defined to be the restriction of the Schwartz kernel of $A$ to $\zf$. 
\end{definition}

The indicial operator $\ib{}(A)$ can be interpreted as a translation-invariant operator on the cylinder $\partial X \times \RR$. As such it is an algebra homomorphism: 
\[\ib{}(PA)=\ib{}(P)\ib{}(A).\]

The compactness of an operator is linked to its indicial operator. 

\begin{proposition}\label{bcomi0}
Suppose that $X$ is a manifold with corners, and $A\in\Psi^m_b(X)$ with $m < 0$. 
Then $A$ is compact on $L^2(X)$ if and only if $\ib{}(A)=0$.
\end{proposition}

When inverting an elliptic partial differential operator in the small $b$-calculus, the  error term will usually have a non-zero indicial operator, and therefore will not be compact. 
In order to obtain an error term whose indicial operator vanishes, 
we have to expand the small $b$-calculus into a bigger calculus, called the full $b$-calculus, in which the Schwartz kernels are permitted to have polyhomogeneous conormal expansions, i.e. expansions in powers and logarithms, at the boundary hypersurfaces $\lbz$, $\zf$ and $\rbz$. 

To define polyhomogeneous cornormal functions, we need the notion of an \emph{index set}. This is  a discrete subset $F\subset\mathbb{C}\times\mathbb{N}_0$ such that every `left segment'
$F\cap\{(z, p): \Re z<N\}$, $N\in\mathbb{R}$ is a finite set. Also, 
it is assumed that if $(z, p)\in F$ and $p\geq q$, $q \in \mathbb{N}$, we also have $(z, q)\in F$.

Given a boundary hypersurface and an index set, we can define polyhomogeneous conormal functions with respect to it. 
They are functions behaving like sums of products of powers and logarithms in one (and hence any) boundary defining function. 
\begin{definition}\label{poly}
Let $X$ be a manifold with boundary and let $H$ be its boundary. 
Given an index set $F$, a smooth function $u$ defined on the interior $X^\circ$ of $X$ 
is called {\sl polyhomogeneous conormal} as it approaches the boundary $H$ with respect to $F$ if, 
on a tubular neighborhood $[0, 1)\times H$ of $H$, one has 
\[u(x, y)\sim\sum_{(z, p)\in F}a_{z, p}(y)x^z\log^p x\]
as $x\rightarrow 0$ with $a_{z, p}$ smooth on $H$.
Here, $\sim$ means that the tail of the series, 
$$
u' = u - \sum_{(z, p)\in F, \Re z \leq B}a_{z, p}(y)x^z\log^p x,
$$
is conormal and vanishes to order $x^{B + \epsilon}$ for some $\epsilon > 0$, in the sense that $|V_1 \dots V_l u' |\leq C x^{B + \epsilon}$ for any finite number of vector fields $V_i$ tangent to $H$ applied to $u'$. 

Given a manifold with corners $X$, and an index family $\mathcal{E}$ for it, i.e.\ an assignment of an index set for each boundary hypersurface, we define polyhomogeneous conormality of $u \in C^\infty(X^\circ)$ by requiring that at each boundary hypersurface, $u$ has an expansion with respect to the corresponding index set with coefficients that are polyhomogeneous conormal on the hypersurface; this sets up an inductive definition. See \cite[Sec. 5.22]{RM93} for details. 
\end{definition}

\begin{definition}[Full $b$-calculus]
The {\sl full $b$-calculus} $\Psi_b^{m,\mathcal{E}}$ on $X$, 
where $m$ is a real number and $\mathcal{E}=(E_{\lbz}, E_{\rbz})$ is an index family for $X^2$, 
is defined as follows. A distribution $u$ on $X_b^2$ is in $\Psi_b^{m,\mathcal{E}}(X)$ 
if and only if $u=u_1+u_2+u_3$ with 
\begin{enumerate}
\item[\rm (i)] $u_1$ is in the small calculus $\Psi_b^m$;
\item[\rm (ii)] $u_2$ is polyhomogeneous conormal with respect to the index family 
$(E_{\lbz}, C^\infty, E_{\rbz})$, where $C^\infty:=\{(n, 0): n\in\mathbb{N}_0\}$ is the $C^\infty$ index set, and the index sets $E_{\lbz}$, $C^\infty$ and $E_{\rbz}$ are assigned to 
the three boundary hypersurfaces $\lbz$, $\zf$, $\rbz$ correspondingly;   
\item[\rm (iii)] $u_3=\beta^*v$, where $\beta: X_b^2\rightarrow X^2$ 
is the blow-down map and $v$ is polyhomogeneous conormal with respect to the index family $\mathcal{E}$.
\end{enumerate}
\end{definition}

\begin{proposition}[{\cite[Prop. 5.46]{RM93}}]\label{bBcomposition}
The full $b$-calculus on $X$ is a two-sided module over the small $b$-calculus, i.e.
\[\Psi_b^{m, \mathcal{E}}(X)\circ\Psi_b^{m'}(X)\subset\Psi_b^{m+m', \mathcal{E}}(X),\]
and
\[\Psi_b^{m'}(X)\circ\Psi_b^{m, \mathcal{E}}(X)\subset\Psi_b^{m+m', \mathcal{E}}(X),\]
where $m, m'\in\mathbb{R}$, and $\mathcal{E}$ is an index family.
\end{proposition}

The reason to introduce the full $b$-calculus is that within it, we can construct parametrices of elliptic $b$-differential operators with compact error term. 
For the proof of the following proposition, see \cite[prop 5.59]{RM93}.
\begin{proposition}\label{compacterror}
Let $P$ be an elliptic $b$-differential operator of order $k$ whose indicial operator $I_b(P)$ is invertible on $L^2(\partial X \times \RR)$. Then there exists $G$ in the full $b$-calculus of order $-k$
 such that the Schwartz kernels of the error terms $E = \Id-PG$ and $E' = \Id-GP$ are smooth across the diagonal, vanish at $\zf$ and are polyhomogeneous conormal at $\lbz, \rbz$ with positive order of vanishing there. This implies that 
$E, E'$ are compact on $L^2(X)$. Necessarily (in view of Proposition~\ref{bcomi0}), we have 
\begin{equation}\label{indicial-inverses}
I_b(G) = I_b(P)^{-1}. 
\end{equation}
\end{proposition}

\subsection{Scattering calculus}

Let $X$ be a manifold with boundary $\partial X$ and with local coordinates $x, y_1,..., y_{d-1}$ near $\partial X$, where
$x$ is a boundary defining function of $\partial X$. 
A smooth vector field $V$ on $X$ is a \emph{scattering vector field} 
if it is $x$ times a $b$-vector field on $X$, ie it has the form 
\[V=a_0x^2\partial_x+a_1x\partial_{y_1}+\cdots+a_{d-1}x\partial_{y_{d-1}},\]
with the coefficients $a_0,...,a_{d-1}$ are smooth functions of $x$ and $y$. Written in terms of $r = x^{-1}$, these take the form
$$
V = -a_0 \partial_r + \frac{a_1}{r} \partial_{y_1} + \dots \frac{a_{d-1}}{r} \partial_{y_{d-1}}.
$$

A scattering differential operator is one that is generated over $C^\infty(X)$ by scattering vector fields. A key example is when $X$ is the radial compactification of $\RR^d$: then any constant coefficient vector field on $\RR^d$ is a scattering vector field viewed on $X$, and therefore any constant coefficient differential operator on $\RR^d$ is a scattering differential operator on $X$. The idea of the scattering calculus is to `microlocalize' this set of differential  operators. 

To define it we first need to blow up the product $X^2$ to produce the scattering double space. This is done in two stages: the first is to create the $b$-double space $X^2_b = [X^2; (\partial X)^2]$ as in the previous subsection. After this blowup, the diagonal lifts to be a product-type submanifold in $X^2_b$, i.e.\ can be expressed as the vanishing of $d$ coordinates in a coordinate system. The second step is to blow up the boundary of the lifted diagonal. The new boundary hypersurfaces so created are denoted $\bfc$ and $\sf$, respectively. 

\begin{proposition}\label{vectorspace}
The interior of the scattering face $\sf$ in the scattering double space $X_{sc}^2$ 
is a bundle over $\partial X$, and each 
fibre $\Omega_y$, $y\in\partial X$, has a natural vector space structure. Moreover, any scattering vector field lifts from either the left or the right factor to be tangent to $\sf$, and to be  a constant coefficient vector field on each fibre. 
\end{proposition}

It is convenient to regard elements of the scattering calculus (defined in the next paragraph) as acting on smooth scattering half-densities, i.e.\ taking the form at the boundary $f |r^{d-1} dr dy|^{1/2}$, $f \in C^\infty(X)$. Thus the Schwartz kernels of such operators will be distributions tensored with the half-density factor
\begin{equation}
\Big| r^{d-1} {r'}^{d-1} dr dy dr' dy' \Big|^{1/2}.
\label{double-sc-half-density}\end{equation}
\begin{definition}[Scattering calculus]
The \emph{scattering calculus} $\Psi_{sc}^{m, l}(X)$ of order $(m, l)$ is defined 
as the set of distributions on $X^2_{sc}$, times \eqref{double-sc-half-density}, satisfying
\begin{enumerate}
\item[\rm (i)] $x^{-l}v$ is conormal of order $m$ with respect to the diagonal (more precisely the diagonal lifted to $X^2_{sc}$)  uniformly up to $\sf$, 
where $x$ is a boundary defining function for $\sf$;
\item[\rm (ii)] $v$ vanishes to infinite order at the other boundary hypersurfaces. 
\end{enumerate}
The order $m$ is called the \emph{differential order} of $v$, and $l$ the \emph{boundary order}.
\end{definition}

\begin{remark} Using the Schwartz kernel theorem, elements of $\Psi_{sc}^{m, l}(X)$ may be interpreted as operators on half-densities on $X$. 
A scattering differential operator of order $m$ acting on half-densities is in $\Psi_{sc}^{m, 0}(X)$. 
\end{remark}

The scattering calculus is closed under composition. 
\begin{proposition}\label{scatteringclosure}\cite[Eqn. 6.12]{RM95}
Let $X$ be a manifold with boundary, and $m,l,m',l'\in\mathbb{R}$, then
\[\Psi_{sc}^{m, l}(X)\circ\Psi_{sc}^{m', l'}(X)\subset\Psi_{sc}^{m+m', l+l'}(X).\]
\end{proposition}

Like Proposition \ref{invert1} on the parametrix constructions under the small $b$-calculus, 
under the scattering calculus we also have a result  analogous to \cite[Theorem 18.1.24]{LH3}.
\begin{proposition}\label{invertsc22}
Suppose that $P\in\Psi_{sc}^{k, 0}(X)$ is elliptic. Then there exists $G\in\Psi_{sc}^{-k, 0}(X)$ such that 
\[PG-\Id,\hspace{2mm}GP-\Id\in\Psi_{sc}^{-\infty, 0}(X).\]
\end{proposition}

Similarly to the case of the indicial operators in Section~\ref{subsec:b}, 
the \emph{normal operator} of $A \in \Psi^{m,0}(X)$, denoted $\norsc(A)$, is defined to be the restriction of the Schwartz kernel of $A$ to the scattering face $\sf$. 
This restriction can be interpreted (in a canonical way) as  a smooth function on $\partial X$ valued in densities on each fibre. 
These densities can be interpreted as convolution operators on functions (or half-densities) on each fibre. Under this interpretation, normal operators can be composed, and the action of taking normal operators is an algebra homomorphism:

\begin{proposition}\cite[Eqn. 5.14]{RM94}
Let $A$ and $B$ be elements of $\Psi^{*, 0}(X)$. Then 
\[\norsc(AB)=\norsc(A)\norsc(B).\]
\end{proposition}

As with the indicial operator, vanishing of the  normal operator is related to compactness:

\begin{proposition} Let $A \in \Psi^{m, 0}(X)$ with $m < 0$. Then $A$ is compact if and only if $\norsc(A)$ vanishes identically. 
\end{proposition}

\begin{remark} Alternatively, we may describe the boundary behaviour in the scattering calculus by taking the fibrewise Fourier transform of each convolution operator, obtaining a family of multipliers; this is known as the normal  or boundary symbol. Composition in terms of the boundary symbol is simply pointwise product. 
\end{remark}

\begin{proposition}\label{sc-invert} If $A \in \Psi^{m,0}(X)$ is elliptic with invertible normal operator, then there exists $B \in \Psi^{-m, 0}(X)$ such that $E = AB - \Id$ is in $\Psi^{-\infty, \infty}(X)$, i.e. its Schwartz kernel is smooth across the diagonal and rapidly vanishing at the boundary of $X^2_{sc}$. In particular, $E$ is compact and hence
$A$ is Fredholm, with parametrix $B$. Necessarily, we have
$$
\norsc{B} = (\norsc{A})^{-1}. 
$$
\end{proposition} 

\begin{proof} See \cite[Section 6]{RM94}. \end{proof}


\section{The blown-up double space}
\label{theblownupspace}
As discussed in the Introduction, we will construct the resolvent kernel $P^{-1} = (H+1)^{-1}$ on a compactified and blown up version of its natural domain $M^2$, using both $b$- and scattering blowups. We start by compactifying $M^2$ in each factor separately, i.e. we pass to the compact space $[0, \infty]_r \times Y \times [0, \infty]_{r'} \times Y$, where $[0, \infty]$ indicates the compactification of $[0, \infty)$ by a point at $r = \infty$, such that $1/r$ is a boundary defining function at $r = \infty$. 
As noted in the Introduction, $r P r$ is an elliptic $b$-differential operator down to $r=0$, while $P$ itself is an elliptic scattering differential operator up to $r = \infty$. Therefore we perform the $b$-blowup at $r = r' = 0$ and the scattering blowups at $r = r' = \infty$. This means that we blow up the corner $r = r' = 0$, the corner $r = r' = \infty$ and finally the boundary of the lifted diagonal $\{ r = r', y = y' \}$ at $r = r' = \infty$. 

We label the boundary hypersurfaces of $[0, \infty]_r \times Y \times [0, \infty]_{r'} \times Y$ by $\lbz$, $\lbi$, $\rbz$ and $\rbi$ according as they arise from $\{ r = 0 \}$, $\{ r = \infty \}$, $\{ r' = 0 \}$, or $\{ r' = \infty \}$, respectively. The new boundary hypersurfaces created by blowup are labelled $\zf$, $\bfc$ and $\sf$, according as they arise from the blowup of $r = r' = 0$, $r = r' = \infty$ or the boundary of the lifted diagonal at $r = r' = \infty$, respectively. The resulting space after the blow-ups at $r=r'=0$ and $r=r'=\infty$ is called \emph{the blown-up space}. See Figure~\ref{domain}.  

We next discuss local coordinates near the various blown up faces. Near $\zf$, local coordinates are $(r/r', r', y, y')$ when $r/r' \leq C$ (that is, away from $\rbz$) and $(r, r'/r, y, y')$ when $r'/r \leq C$ (that is, away from $\lbz$). Near $\bfc$ and away from $\sf$ the situation is similar: coordinates are $(r'/r, {r'}^{-1}, y, y')$ for $r'/r \leq C$ and $(r/r', r^{-1}, y, y')$ for $r/r' \leq C$. 
Near the interior of $\sf$, coordinates are $(r-r', r(y-y'), y, r^{-1})$. In the case that $M$ is Euclidean space $\RR^d$, with Euclidean coordinate $z$, then $z-z'$ is a linear coordinate on each fibre of $\sf$ (cf. Proposition~\ref{vectorspace}). 
In particular, the diagonal is defined by $r/r' = 1, y=y'$ for small $r$ (that is, away from $\sf$ ) and $r -r' = 0$, $r(y-y') = 0$ or $r -r' = 0$, $r'(y-y') = 0$ for large $r$ (that is, away from $\zf$). 
The following result about the diagonal will be useful later.  

\begin{proposition}\label{quadraticdefining}
Let $\varphi:[0, \infty)\rightarrow [0, 1]$ be an increasing smooth function such that 
$\varphi(x)=x$ for $x\in[0, \frac{1}{2}]$ and $\varphi(x)=1$ for $x\in [1, \infty)$. 
Then the function 
\[a_{\rm diag}(z, z')=\frac{\dist(z, z')^2}{\varphi^2(r')},\]
where $z=(r, y)$ and $z'=(r', y')$, 
is a quadratic defining function for the diagonal in the blown-up space; that is, $a_{\rm diag} \geq 0$, the diagonal lifted to the blown up space is given by
$\{ a_{\rm diag} = 0 \}$, and the Hessian of $a_{\rm diag}$ in directions normal to the diagonal is positive definite. 
\end{proposition}
\begin{proof}
The formula for the distance on a metric cone is given by 
\begin{equation}
\dist(z, z')^2= \begin{cases} r^2+r'^2-2rr'\cos\big(\dist_Y(y, y')\big), \quad \dist_Y(y,y') \leq \pi \\
(r+r')^2, \quad \dist_Y(y,y') \geq \pi. \end{cases}
\label{distanceonm}\end{equation}
(The second line is because when $\dist_Y(y,y') \geq \pi$ the fastest way to get from $(r,y)$ to $(r', y')$ is to go straight to the cone point and back out again.) So  near the diagonal we have
\begin{equation}
\begin{split}
\dist(z, z')^2&=(r-r')^2+2rr'\bigg(1-\cos\big(d_Y(y, y')\big)\bigg)\\
&=(r-r')^2+rr'\bigg(\dist_Y(y,y')^2+O\big(\dist_Y(y,y')^4\big)\bigg).
\end{split}
\end{equation}
Near the $\sf$-face, we have 
\[a_{\rm diag}(z, z')=\dist(z, z')^2=(r-r')^2+rr'\bigg(\dist_Y(y,y')^2+O\big(\dist_Y(y,y')^4\big)\bigg),\]
which is good for a quadratic defining function for the diagonal. 
To see that we recall from the discussion before this proposition that near $\sf$ the diagonal is defined by 
$r'-r=0$, and $r(y-y')=0$ or $r'(y-y')=0$, and we 
also recall the standard fact that $\dist_Y(y, y')^2$ 
is a quadratic defining function for the diagonal of $Y^2$ for any closed Riemannian manifold $Y$.

Near the $\zf$-face, we have 
\[a_{\rm diag}(z, z')=\frac{\dist(z, z')^2}{r'^2}=(\frac{r}{r'}-1)^2+\frac{r}{r'}\bigg(\dist_Y(y,y')^2+O\big(\dist_Y(y,y')^4\big)\bigg),\]
which is again good for a quadratic defining function for the diagonal, 
as here the diagonal is instead defined by $\frac{r}{r'}=1$ and $y=y'$. 
\end{proof}

\subsection{Densities on the blown-up space}
By a  smooth $b$-half-density on the blown-up space we mean a half-density of  the form 
\[u(r, r', y, y')\bigg|\frac{dr}{r}\frac{dr'}{r'}dydy'\bigg|^{\frac{1}{2}},\]
where $u$ is smooth.  (This is perhaps misleading since it is only a $b$-half density in the usual sense away from $\sf$. However, we shall only use this when either $r$ or $r'$ is small, in which case it certainly is a $b$-half density.) 
Let $x=\frac{1}{r}, x'=\frac{1}{r'}$. Then by a smooth scattering-half-density we mean a density of the form, 
\[v(x, x', y, y')\bigg|\frac{dxdx'dydy'}{x^{d+1}x'^{d+1}}\bigg|^{\frac{1}{2}},\]
where $v$ is smooth. In terms $r$ and $r'$ it becomes, 
\[v(r, r', y, y')\big|r^{d+1}r'^{d+1}d(\frac{1}{r})d(\frac{1}{r'})dydy'\big|^{\frac{1}{2}}
=v(r, r', y, y')\big|r^{d-1}r'^{d-1}drdr'dydy'\big|^{\frac{1}{2}}.\]

The scattering half-density 
$|r^{d-1}r'^{d-1}drdr'dydy'|^{\frac{1}{2}}$
is a bounded nonzero multiple of the Riemannian half-density. 
We will usually consider the resolvent $P^{-1}$ as acting on Riemannian half-densities, in which case the kernel of $P^{-1}$ itself is a Riemannian (distributional) half-density on the blown-up space. 
However, when we study the properties of a kernel near the $\zf$-face, we write it as a $b$-half-density; this is more natural in view of the fact that we use the $b$-calculus near $\zf$. 


\section{Resolvent Construction}
\label{chapter5}

\subsection{The operator $H$}
Let $M$ be the metric cone over $(Y, h)$. 
The Laplacian on the cone $M$ expressed in polar coordinates is 
\begin{equation*}
-\partial^2_r-\frac{d-1}{r}\partial_r+\frac{1}{r^2}\Delta_Y,
\end{equation*}
where $\Delta_Y$ is the Laplacian on $Y$. This operator is positive and symmetric on the domain $C_c^\infty(Y \times (0, \infty))$, i.e. smooth functions supported away from the cone tip.
The operator $\Delta$ is defined to be the Friedrichs extension of this symmetric operator. 

For any function $V_0: Y\rightarrow\mathbb{C}$, we define the operator 
\[H_{V_0}=\Delta+\frac{V_0(y)}{r^2}.\]
This is a natural class of operators: as both $\Delta$ and $\frac{V_0(y)}{r^2}$ are homogeneous of degree $-2$, 
the operator $H_{V_0}$ has the same homogeneity. 
For simplicity of notation, we write $H_{V_0}$ simply as $H$. 
The following proposition tells us for which $V_0$ is the operator $H$ positive. 

\begin{proposition}\label{positive}
Suppose that $\Delta_Y+V_0(y)+(\frac{d-2}{2})^2$ is a positive operator on $L^2(Y)$. Then the operator $H$ is also positive.
\end{proposition}
\begin{proof}
We work in polar coordinates, 
and consider the isometry $U: L^2(M; r^{d-1}drdy)\rightarrow L^2(M; r^{-1}drdy)$ defined by
\begin{equation}\label{Udefinition}
Uf=r^{\frac{d}{2}}f.
\end{equation}
Now for $f\in L^2(M; r^{-1}drdy)$, we compute
\begin{equation*}
\begin{split}
U HU^{-1}f
=r^{\frac{d}{2}}\bigg(-\partial^2_r-\frac{d-1}{r}\partial_r+\frac{1}{r^2}\Delta_Y+\frac{V_0(y)}{r^2}\bigg)r^{-\frac{d}{2}}f.\\
=\bigg(\frac{d(d-4)}{4}+V_0(y)\bigg)\frac{1}{r^2}f+\frac{1}{r}\partial_rf-\partial_r^2f+\frac{1}{r^2}\Delta_Yf.
\end{split}
\end{equation*}
Therefore we have,
\begin{equation*}
\begin{split}
&\Bigg(\frac{1}{r}\bigg(-(r\partial r)^2+\Delta_{Y}+\Big(\frac{d-2}{2}\Big)^2+V_0(y)\bigg)\frac{1}{r}\Bigg)f\\
=&-\partial_r\Big(r\partial_r(\frac{1}{r}f)\Big)+\frac{1}{r^2}\Delta_Yf+\bigg(\Big(\frac{d-2}{2}\Big)^2+V_0(y)\bigg)\frac{1}{r^2}f\\
=&\partial_r(\frac{1}{r}f)-\partial_r^2f+\frac{1}{r^2}\Delta_Yf+\bigg(\Big(\frac{d-2}{2}\Big)^2+V_0(y)\bigg)\frac{1}{r^2}f\\
=&-\frac{1}{r^2}f+\frac{1}{r}\partial_r f-\partial_r^2f+\frac{1}{r^2}\Delta_Yf+\bigg(\Big(\frac{d-2}{2}\Big)^2+V_0(y)\bigg)\frac{1}{r^2}f\\
=&\bigg(\frac{d(d-4)}{4}+V_0(y)\bigg)\frac{1}{r^2}f+\frac{1}{r}\partial_r f-\partial_r^2f+\frac{1}{r^2}\Delta_Yf\\
=&UHU^{-1}f.\\
\end{split}
\end{equation*}
We have established
\begin{equation}\label{Utransformation}
UHU^{-1}=\frac{1}{r}\bigg(-(r\partial_r)^2+\Delta_Y+V_0(y)+\Big(\frac{d-2}{2}\Big)^2\bigg)\frac{1}{r}.
\end{equation}
Make a substitution $s=\ln r$, then the space $L^2(M; r^{-1}drdy)$ becomes $L^2(M; dsdy)$, 
and we have
\[UHU^{-1}=e^{-s}\bigg(-\partial_s^2+\Delta_Y+V_0(y)+\Big(\frac{d-2}{2}\Big)^2\bigg)e^{-s}.\]
From here we can clearly see that the operator $H$ is positive if $\Delta_Y+V_0(y)+(\frac{d-2}{2})^2>0$. 
This completes the proof.
\end{proof}

\begin{remark}
Note that as we have $d\geq 3$, 
the condition $\Delta_Y+V_0(y)+(\frac{d-2}{2})^2>0$ means that 
the potential $V=\frac{V_0}{r^2}$ is allowed to be ``a bit negative''.
\end{remark}

\subsection{The Riesz transform $T$}
Suppose we have a function $V_0$ on $Y$ 
which satisfies $\Delta_Y+V_0(y)+(\frac{d-2}{2})^2>0$. Then 
the Riesz transform $T$ with the inverse square potential $V=\frac{V_0}{r^2}$ is defined to be
\[T=\nabla H^{-\frac{1}{2}},\]
where the size of the derivatives are measured using the cone metric $g$, i.e.\ derivatives of bounded length are given by $(\partial_r, r^{-1} \partial_{y_i})$. 

Our aim is to find out the precise range of $p$ for which the Riesz transform $T$ is bounded on $L^p(M)$. Following \cite{CCH} and  \cite{GH}, we do this using a `resolvent approach' as opposed to the more common `heat kernel approach'. 
Using functional calculus, we have the following expression,
\[T=\frac{2}{\pi}\int_0^\infty \nabla(H+\lambda^2)^{-1}d\lambda.\]
We see from this equation that in order to understand $T$, 
we need to know the properties of $(H+\lambda^2)^{-1}$. 
Because $H$ is homogeneous of degree $-2$, we only need to compute $(H+1)^{-1}$, then use scaling.
Let $P=H+1$; we will proceed to study $P^{-1}$.

\subsection{A formula for the resolvent}
\label{formula}
We now proceed to find an explicit formula for $P^{-1}$. 
However as we will discuss later, the formula has good convergence properties in only certain regions of the blown-up space.
From Equation (\ref{Utransformation})  we have 
\begin{equation*}
P=H+1=r^{-\frac{d}{2}-1}\big(-(r\partial_r)^2+\Delta_Y+V_0(y)+r^2+\Big(\frac{d-2}{2}\Big)^2\big)r^{\frac{d}{2}-1}.
\end{equation*}
Let $P'$ denote the differential operator consisting  of the terms in the middle. That is, 
\begin{equation}\label{Ptildedefinition}
P'=-(r\partial_r)^2+\Delta_Y+V_0(y)+r^2+\Big(\frac{d-2}{2}\Big)^2.
\end{equation}
We take $P'$ to act on half-densities, using the flat connection that annihilates the Riemannian half-density $|r^{d-1} dr dh|^{1/2}$ on $M$. Now let $\tilde P$ be the differential operator given by the same expression \eqref{Ptildedefinition}, but endowed with the flat connection on half-densities annihilating the $b$-half density $|dr/r dh|^{1/2}$. Since $U$ maps this $b$-half density to the Riemannian half-density, these two differential operators are related by 
\begin{equation}
\tilde P = U^{-1} P' U.
\end{equation} 
Therefore, 
\begin{equation}\label{pptilde}
P=r^{-1}\tilde{P}r^{-1}. 
\end{equation}
Since $P$ is self-adjoint, 
Equation (\ref{pptilde}) shows that $\tilde{P}$ 
is also self-adjoint. (Note that for operators on half-densities there is an invariant notion of self-adjointness, since the inner product on half-densities is invariantly defined.) 
Denote $G=P^{-1}$, $\tilde{G}=\tilde{P}^{-1}$; the Schwartz kernels of $G$ and $\tilde{G}$ are related by
\begin{equation}\label{GGtilde}
G=  r r'\tilde{G}.
\end{equation}
(Again, we emphasize that this is an identity involving half-densities: if we write $G = K |(rr')^{d-1} dr dr' dh dh'|^{1/2}$ and $\tilde G = \tilde K |(rr')^{-1} dr dr' dhdh'|^{1/2}$ then we have 
\begin{equation}\label{KKtilde}
K=  (r r')^{1-d/2}\tilde{K}.)
\end{equation}
So we just need to determine $\tilde{G}$, then Equation (\ref{GGtilde}) gives us $G$. 

We now proceed to work out an expression for $\tilde{G}$. 
Let $(\mu_j^2, u_j)$ be the eigenvalues and corresponding $L^2$-normalized eigenfunctions of 
the positive operator $\Delta_Y+V_0(y)+(\frac{d-2}{2})^2$. 
We also let $\Pi_j$ denote the projection onto the $u_j$-eigenspace. 
Then we have
\begin{equation}\label{ptilde}
\tilde{P}=\sum_j\Pi_j \tilde{T}_j,
\end{equation}
and 
\[ \Id=\sum_{j}\delta(\frac{r}{r'}-1)\Pi_j,\]
where 
\begin{equation}\label{ttilde}
\tilde{T}_j=-(r\partial_r)^2+r^2+\mu_j^2=-r^2\partial^2_r-r\partial_r+\mu_j^2.
\end{equation} 

As in \cite{GH}, the kernel of the inverse of $\tilde{T}_j$ is written in terms of  modified Bessel functions $I_{\mu_j}(r)$ and $K_{\mu_j}(r)$ (see \cite[Sec. 9.6]{AS}) in the form 
\[\tilde{T}^{-1}_j(r,r')=
\begin{cases}
I_{\mu_j}(r)K_{\mu_j}(r')\big|\frac{dr}{r}\frac{dr'}{r'}\big|^{\frac{1}{2}},\hspace{3mm}r<r',\\
K_{\mu_j}(r)I_{\mu_j}(r')\big|\frac{dr}{r}\frac{dr'}{r'}\big|^{\frac{1}{2}},\hspace{3mm}r>r',
\end{cases}\]
We know that
\[\tilde{G}=\sum_j\Pi_j \tilde{T}^{-1}_j,\]
hence in terms of the kernels, we have
\begin{equation}\label{infiniteseries}
\tilde{G}(r,r', y, y')= 
\begin{cases}
\sum_j u_j(y)\overline{u_j(y')}I_{\mu_j}(r)K_{\mu_j}(r')\big|\frac{dr}{r}\frac{dr'}{r'}dhdh'\big|^{\frac{1}{2}},\hspace{3mm}r<r',\\
\sum_j u_j(y)\overline{u_j(y')}K_{\mu_j}(r)I_{\mu_j}(r')\big|\frac{dr}{r}\frac{dr'}{r'}dhdh'\big|^{\frac{1}{2}},\hspace{3mm}r>r',
\end{cases}
\end{equation}
where $dh$ denotes the Riemannian density with respect to the metric on $Y$. 
While this is an exact expression for $\tilde G$, it is not a very useful expression near the diagonal, as it has poor convergence properties. Therefore we shall glue it together with a pseudodifferential-type parametrix in order to determine its properties close to the diagonal. However, sufficiently far from the diagonal, the series has very good convergence. We proceed to show this. 

\subsection{Convergence of the formula}
By the symmetry of \eqref{infiniteseries}, it suffices to consider the region  $\{ r< r' \}$; here we work with the sum, 
\begin{equation}
\sum_j u_j(y)\overline{u_j(y')}I_{\mu_j}(r)K_{\mu_j}(r').
\label{convsum}\end{equation}
From \cite[Sec. 9.6]{AS}, we have representations 
\[I_{\mu}(r)=\frac{2^{-\mu}r^{\mu}}{\pi^{\frac{1}{2}}\Gamma(\mu+\frac{1}{2})}\int_{-1}^1(1-t^2)^{\mu-\frac{1}{2}}e^{-rt}dt,\]
and
\[K_{\mu}(r')=\frac{\pi^{\frac{1}{2}}2^{-\mu}r'^{\mu}}{\Gamma(\mu+\frac{1}{2})}\int_1^\infty e^{-r't}(t^2-1)^{\mu-\frac{1}{2}}dt.\]

We now estimate each of these integrals in a way that is uniform as $\mu \to \infty$. When $r \leq 1$, the integral in the expression for $I_\mu$ is uniformly bounded in $\mu > 0$, and hence we see that 
\begin{equation}
 \big| I_{\mu}(r) \big| \leq C \frac{2^{-\mu} r^\mu}{\Gamma(\mu + 1/2)} \text{ when } r \leq 1,
\label{Ismall}\end{equation}
where $C$ is independent of $r$ and $\mu$. On the other hand, for $r \geq 1$, we estimate $e^{-r} r^{1/2}  I_\mu(r)$:
\begin{align*}
e^{-r} r^{1/2}  I_\mu(r) &= \frac{2^{-\mu}r^{\mu + 1/2}}{\pi^{\frac{1}{2}}\Gamma(\mu+\frac{1}{2})}\int_{-1}^1(1-t^2)^{\mu-\frac{1}{2}}e^{-r(t+1)}dt \\
&\leq C \frac{2^{-\mu}r^{\mu + 1/2}}{\Gamma(\mu+\frac{1}{2})} \int_{-1}^{1} e^{-r(t+1)} \, dt \\
&\leq C \frac{2^{-\mu}r^{\mu + 1/2}}{\Gamma(\mu+\frac{1}{2})} \int_0^{2r} e^{-t} \, \frac{dt}{r} \\
&\leq C \frac{2^{-\mu}r^{\mu - 1/2}}{\Gamma(\mu+\frac{1}{2})}
\end{align*}
with $C$ independent of $\mu$, for $\mu \geq 1/2$. 
 This gives rise to an estimate of the form 
\begin{equation}
\big| I_{\mu}(r) \big| \leq C     \frac{2^{-\mu } r^{\mu - 1}e^{r}}{\Gamma(\mu + 1/2)}  \text{ when } r \geq 1. 
\label{Ilarge}\end{equation}

We next estimate $K_\mu$ in a similar way. For $r \leq 1$, we estimate 
\begin{equation}\begin{aligned}
K_{\mu}(r)&=\frac{\pi^{\frac{1}{2}}2^{-\mu}r^{\mu}}{\Gamma(\mu+\frac{1}{2})}\int_1^\infty e^{-rt}(t^2-1)^{\mu-\frac{1}{2}}\, dt \\
&= \frac{\pi^{\frac{1}{2}}2^{-\mu}r^{\mu}}{\Gamma(\mu+\frac{1}{2})}\int_0^\infty e^{-rt}t^{2\mu-1}\, dt \\
& \leq C \frac{2^{-\mu}r^{\mu}}{\Gamma(\mu+\frac{1}{2})}\int_0^\infty e^{-t}t^{2\mu-1} r^{-2\mu}\, dt \\
&= C \frac{2^{-\mu}r^{-\mu} \Gamma(2\mu)}{\Gamma(\mu+\frac{1}{2})}.
\end{aligned}\label{Ksmall}\end{equation}
On the other hand, for $r \geq 1$, we compute 
\begin{align*}
e^{r} r^{1/2}  K_\mu(r) &= \frac{\pi^{\frac{1}{2}}2^{-\mu}r^{\mu + 1/2}}{\Gamma(\mu+\frac{1}{2})}\int_1^\infty e^{-r(t-1)}(t^2-1)^{\mu-\frac{1}{2}}\, dt \\
&= \frac{\pi^{\frac{1}{2}}2^{-\mu}r^{\mu + 1/2}}{\Gamma(\mu+\frac{1}{2})}\int_0^\infty e^{-t}(t(2r+t))^{\mu-\frac{1}{2}} r^{-2\mu}\, dt \\
\end{align*}
where we made a substitution $t \to r(t-1)$ in the integral. We now estimate 
$$
(2r+t)^{\mu-\frac{1}{2}} \leq 2^{\mu - 1/2} \max \Big( (2r)^{\mu - 1/2},  t^{\mu - 1/2} \Big)
$$
which gives rise to an estimate 
\begin{equation}
\Big|   K_\mu(r)  \Big| \leq C  \frac{e^{-r}  r^{-\mu}}{\Gamma(\mu+\frac{1}{2})} \max \Big( (2r)^{\mu - 1/2} \Gamma(\mu + 1/2), \Gamma(2\mu) \Big).
\label{Klarge}\end{equation}
Now to absorb the factor $r^{\mu - 1/2}$ in the first argument of the maximum function, we sacrifice half of our exponential decay: we estimate $e^{-r/2} r^{\mu - 1/2}$ by bounding it by the value where it achieves its maximum in $r$, which is when $r = 2\mu - 1$:
$$
e^{-r/2} r^{\mu - 1/2} \leq e^{-(2\mu - 1)/2} (2\mu - 1)^{\mu - 1/2} \leq C e^{-\mu} 2^\mu \mu^{(\mu - 1/2)} \leq C 
2^{\mu } \Gamma(\mu).
$$
Then we can use this in \eqref{Klarge} to estimate 
\begin{equation}
\Big|   K_\mu(r)  \Big| \leq C  \frac{e^{-r/2}  r^{-\mu }}{\Gamma(\mu+\frac{1}{2})} \max \Big( 2^{2\mu } \Gamma(\mu) \Gamma(\mu + 1/2), \Gamma(2\mu) \Big).
\label{Klarge2}\end{equation}
Finally using the identity (see \cite[6.1.18]{AS})
\begin{equation}
\Gamma(2\mu)=\frac{2^{2\mu-1}}{\sqrt{\pi}}\Gamma(\mu)\Gamma(\mu+\frac{1}{2})
\label{Gidentity}\end{equation}
we obtain 
\begin{equation}
\Big|   K_\mu(r)  \Big| \leq C  e^{-r/2}  r^{-\mu }  2^{2\mu } \Gamma(\mu) .
\label{Klarge3}\end{equation}

Hence, when $r' \geq 4r$,  $I_{\mu}(r)K_{\mu}(r')$ is bounded above by
\begin{equation}
\begin{cases}
C  \big( \frac{r}{r'} \big)^\mu  \frac{ 2^{-2\mu} \Gamma(2\mu)}{\big(\Gamma(\mu + 1/2)\big)^2 }, \quad 0 \leq r \leq r' \leq 1 \\
  C 2^{\mu } \big( \frac{r}{r'} \big)^\mu e^{-r'/2} 
    \frac{ \Gamma(\mu)}{\Gamma(\mu + 1/2)},
  \quad r \leq 1 \leq r'  \\
  C     2^{\mu } \big( \frac{r}{r'} \big)^{\mu } e^{-r'/4}  \frac{\Gamma(\mu)}{\Gamma(\mu + 1/2)}, \quad 1 \leq r \leq r' . 
 \end{cases}
\end{equation}
We emphasize that the constant $C$ is independent of $\mu \geq 1/2$, $r$ and $r'$ here. Noting that the combination of $\Gamma$ factors is uniformly bounded in each case (using \eqref{Gidentity} again), 
we find that for $r' \geq 4r$,  $I_{\mu}(r)K_{\mu}(r')$ is bounded above by
\begin{equation}
\begin{cases}
C  \big( \frac{r}{r'} \big)^\mu, \quad 0 \leq r \leq r' \leq 1 \\
  C  \big( \frac{2r}{r'} \big)^\mu e^{-r'/2},
  \quad r \leq 1 \leq r'  \\
  C    \big( \frac{2r}{r'} \big)^{\mu } e^{-r'/4}, \quad 1 \leq r \leq r' . 
 \end{cases}
\label{expvanishing}\end{equation}

By H{\" o}rmander's $L^\infty$-estimate, see \cite{LH2}, we know that $||u_j||_\infty\leq C\mu_j^{\frac{d-1}{2}}$. 
Therefore each term in the series is bounded above by $C\mu_j^{d-1}(\frac{2r}{r'})^{\mu_j}e^{-\frac{r'}{4}}$. 
To continue the discussion on convergence, we need the following lemma. 

\begin{lemma}\label{sumlemma}
Suppose that $\mu_j^2$ are the eigenvalues of $\Delta_Y+V_0(y)+(\frac{d-2}{2})^2$, 
then for any $0<\beta<1$, and any $M, N\geq 0$, the sum 
\[\sum_{\mu_j\geq M}\mu_j^N\alpha^{\mu_j-M}\]
converges for all $0<\alpha\leq\beta$, and it is bounded uniformly in $\alpha$.
\end{lemma}
\begin{proof}
Note that for any $\mu_j\geq 2M$, we have 
\[\mu_j-M=M+(\mu_j-2M)\geq M+\Big(\frac{\mu_j-2M}{2}\Big)=\frac{\mu_j}{2}.\]
Therefore,
\[\sum_{\mu_j\geq 2M}\mu_j^N\alpha^{\mu_j-M}\leq\sum_{\mu_j\geq 2M} \mu_j^N\alpha^{\frac{\mu_j}{2}}\leq\sum_{\mu_j\geq 2M}\mu_j^N\beta^{\frac{\mu_j}{2}}.\]
There is an integer $N_1(\beta, N)>2M$ such that for all $j\geq N_1(\beta, N)$, $j^N\leq\beta^{-\frac{j}{4}}$. It follows that
\begin{equation}
\begin{split}
\sum_{\mu_j\geq 2M}\mu_j^N\beta^{\frac{\mu_j}{2}}
&\leq\sum_{2M\leq\mu_j<N_1(\beta, N)}\mu_j^N\beta^{\frac{\mu_j}{2}}
+\sum_{\mu_j\geq N_1(\beta, N)}\beta^{-\frac{\mu_j}{4}}\beta^{\frac{\mu_j}{2}}\\
&\leq\big|\{j: \mu_j<N_1(\beta, N)\}\big|N_1(\beta, N)^N
+\sum_{\mu_j\geq N_1(\beta, N)}\beta^{\frac{\mu_j}{4}}\\
&\leq CN_1(\beta, N)^{d+N-1}
+\sum_{\mu_j\geq N_1(\beta, N)}\beta^{\frac{\mu_j}{4}},
\end{split}
\end{equation}
where the constant $C>0$ comes from the Weyl's estimate, which states that for any $\mu>1$, we have
\begin{equation}\label{Weyl}
\big|\{j: \mu_j\leq\mu\}\big|\leq C\mu^{d-1}.
\end{equation}
We continue to estimate the part of summation greater than $N_1(\beta, N)$. 
An implication of (\ref{Weyl}) is, for any $j\in\mathbb{N}$, we have 
\[\mu_j\geq\Big(\frac{j}{C}\Big)^{\frac{1}{d-1}}.\]
Therefore,
\[\sum_{\mu_j\geq N_1(\beta, N)}\beta^{\frac{\mu_j}{4}}\leq\sum_{\mu_j\geq N_1(\beta, N)}\beta^{\frac{1}{4}(\frac{j}{C})^{\frac{1}{d-1}}}
\leq\sum_{j\geq 0}\beta^{\frac{1}{4}(\frac{j}{C})^{\frac{1}{d-1}}}\]
There is $N_2(\beta, C)\in\mathbb{N}$ such that for all $j\geq N_2(\beta, C)$, 
we have $\frac{1}{4}(\frac{j}{C})^{\frac{1}{d-1}}>\log_\gamma j$, where $\gamma=\beta^{-\frac{1}{2}}>1$. 
Then
\begin{equation}
\begin{split}
\sum_{j\geq 1}\beta^{\frac{1}{4}(\frac{j}{C})^{\frac{1}{d-1}}}
&\leq\sum_{0 \leq j<N_2(\beta, C)}\beta^{\frac{1}{4}(\frac{j}{C})^{\frac{1}{d-1}}}+\sum_{j\geq N_2(\beta, C)}\beta^{\log_\gamma j}\\
&\leq N_2(\beta, C)+\sum_{j\geq N_2(\beta, C)}j^{-2}\\
&\leq N_2(\beta, C)+\frac{\pi^2}{6}.\\
\end{split}
\end{equation}
The remaining part of the summation is from $M$ to $2M$, 
\[\sum_{M\leq\mu_j<2M}\mu_j^N\alpha^{\mu_j-M}\leq\big|\{j: \mu_j<2M\}\big|(2M)^N
\leq C(2M)^{d-1}(2M)^N=C(2M)^{d+N-1}.\]
Bringing all the parts together, we have
\begin{equation}
\begin{split}
\sum_{\mu_j\geq M}\mu_j^N\alpha^{\mu_j-M}
\leq C(2M)^{d+N-1}+CN_1(\beta, N)^{d+N-1}+N_2(\beta, C)+\frac{\pi^2}{6}<\infty.
\end{split}
\end{equation}
Note the finite constant depends on $M, N, C, \beta$ but not $\alpha$, therefore we have uniform boundedness in $\alpha$.
\end{proof}

\begin{figure}[h!]
\centering
\includegraphics[scale=0.35]{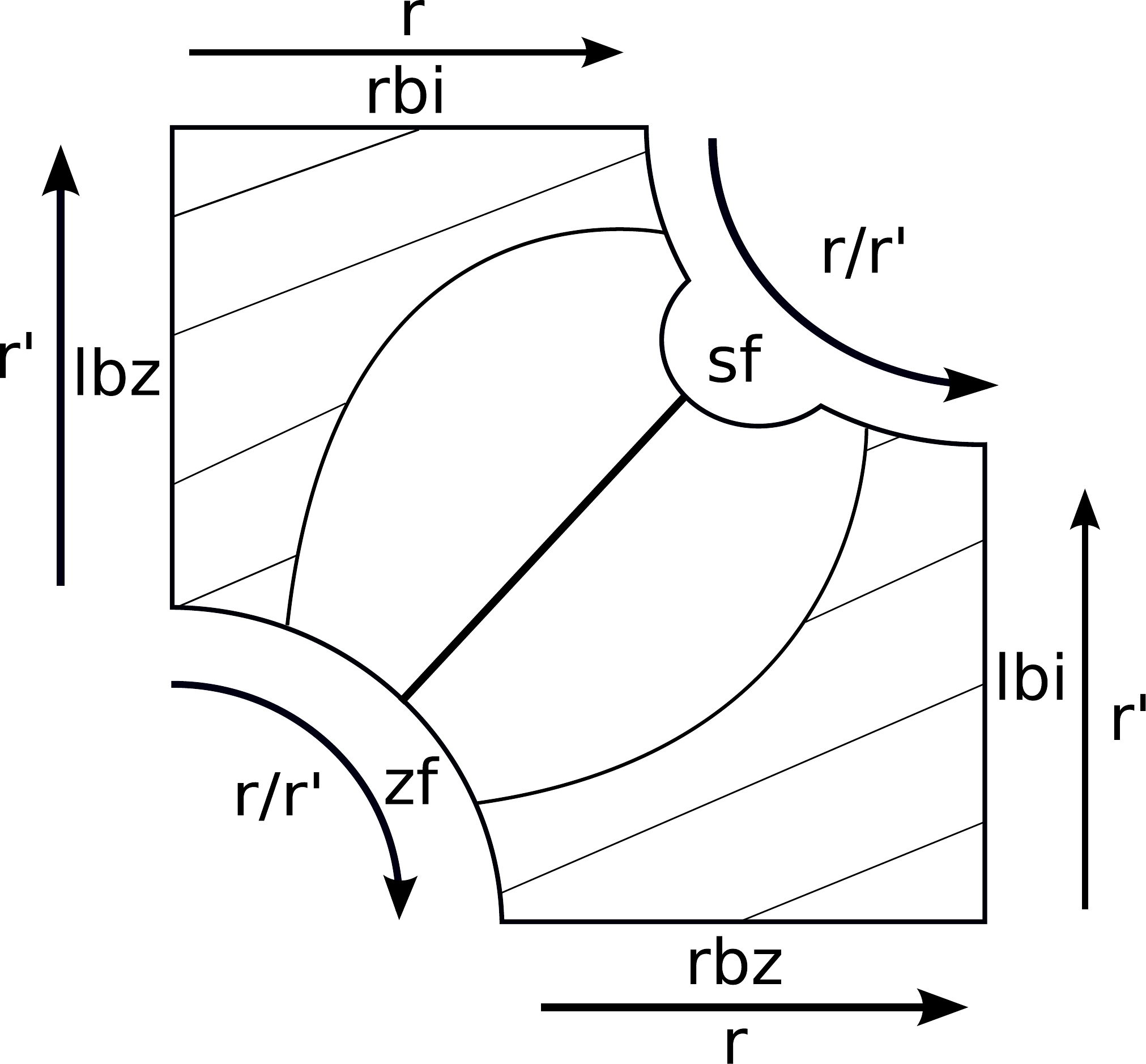}
\caption{Support of $\tilde{G}_f$}
\label{supportg0}
\end{figure}

\begin{proposition}\label{prop:phg}
The expansion \eqref{convsum} is polyhomogeneous conormal at $\lbz$. 
\end{proposition} 

\begin{proof}
Since the functions $I_\mu(r)$ and $K_\mu(r)$ have expansions in powers at $r = 0$ (including logarithms in the case of $K_\mu$ when $\mu$ is an integer), the individual terms in the series are polyhomogeneous conormal. So consider the tail of the series. Lemma~\ref{sumlemma} implies that the sum of the tail of the series, that is over $\mu_j \geq M$ is bounded by  $Cr^M e^{-r'/4}$ for small $r$. We can apply the same argument to derivatives of the series. In fact, the derivatives of $I_\mu$ and $K_\mu$ can be treated as above, showing that for $r \leq 1$, $\mu^{-k} (r\partial_r)^k I_\mu$ and $\mu^{-k} (r\partial_r)^k K_\mu$, and for $r \geq 1$,  $\mu^{-k} \partial_r^k I_\mu$ and $\mu^{-k} \partial_r^k K_\mu$ satisfy the same estimates as $I_\mu$ and $K_\mu$. Moreover, 
we have a  H\"ormander estimate $\| \nabla^{(k)} u_j \|_{\infty } \leq C_k \mu_j^{(d-1)/2 + k}$ for derivatives of $u_k$. Thus derivatives only give us extra powers of $\mu$, which are harmless as  Lemma~\ref{sumlemma} applies with arbitrary powers of $\mu$. 
\end{proof}

Proposition~\ref{prop:phg} implies, in particular, that $\tilde G$ decays exponentially, with all its derivatives, as $r'\rightarrow\infty$, 
ie when approaching the boundary $\rbi${}. 
Similarly in the region $\frac{r}{r'}\geq 4$, as $r\rightarrow\infty$, ie when approaching $\lbi${}, the kernel is also exponentially decreasing.
Therefore we cut off $\tilde{G}$ to restrict it away from the $r=r'$ to obtain 
a well defined operator $\tilde{G}_f$ with the kernel
\begin{equation}\label{initial}
\tilde{G}_f(r, r', y, y')=\tilde{G}(r, r', y, y')\bigg(\chi\Big(\frac{4r}{r'}\Big)+\chi\Big(\frac{4r'}{r}\Big)\bigg).
\end{equation}
Here $\chi$ is a smooth cutoff function $\chi:[0,\infty)\rightarrow[0, 1]$ 
such that $\chi\big([0,1/2]\big)=1$ and $\chi\big([1,\infty)\big)=0$.
Thus the support of $\tilde{G}_f$ is contained in $\{ r/r' \leq 1/4 \} \cup \{ r/r' \geq 4 \}$,  as illustrated in Figure \ref{supportg0}. (The subscript `f' stands for `far from the diagonal'.)

At last, similar to (\ref{GGtilde}), we define
\begin{equation}\label{GG0tilde}
G_f= r r'  \tilde{G}_f.
\end{equation}

\subsection{Near Diagonal}\label{snd}

The formula obtained in the previous section doesn't converge near the diagonal, so in this 
section we construct an operator $G_{nd}$, which is good near the diagonal.
The subscript $nd$ means ``near diagonal''. 

Near the $\zf$-face we consider the $b$-elliptic operator $\tilde{P}$. 
In order to keep it away from the $\sf$-face, we multiply it with a cutoff function, so we consider $\tilde{P}\chi(r)$, 
where $\chi:[0, \infty)\rightarrow[0, 1]$ is a smooth cutoff  function as above. By the ellipticity of $\tilde{P}$ near the $\zf$-face, 
and by Proposition \ref{compacterror}, there is $\tilde{G}_{nd}^{zf}$ in the full $b$-calculus such that 
\begin{equation}\label{errorzf}
\tilde{P}\tilde{G}_{nd}^{zf}\chi(r)=\chi(r)+ \tilde E_{zf},
\end{equation}
where $\tilde E_{zf}$ is smooth across the diagonal and vanishes to first order at $\zf$ (as a $b$-half density). 
Let 
\begin{equation}G_{nd}^{zf}= r r' \tilde{G}_{nd}^{zf},
\label{Gzfdefinition}\end{equation}
then we have 
\[PG_{nd}^{zf}\chi(r)=\chi(r)+E_{zf}\]
where $E_{zf} = (r'/r) \tilde E_{zf}$ is smooth across the diagonal and vanishes to first order at $\zf$ as a $b$-half density.

Near the $\sf$-face the operator $P$ is elliptic in the scattering calculus. 
We multiply it with $1-\chi(r)$ to keep it away from the $zf$-face, ie we consider the operator $P\big(1-\chi(r)\big)$. 
Since $P\big(1-\chi(r)\big)$ is elliptic near the $\sf$-face, and its normal operator $\Delta_{\RR^n} + 1$ is invertible, 
by Proposition \ref{sc-invert}, there is $G_{nd}^{sf}$ in the scattering calculus such that 
\[PG_{nd}^{sf}\big(1-\chi(r)\big)=1-\chi(r)+E_{sf},\]
where the error term $E_{sf}$ is smooth across the diagonal and vanishes to infinite order at the boundary of the blown up space. 

Now we define $G_{nd}$ by 
\[G_{nd}=(G_{nd}^{zf}+G_{nd}^{sf})\Big( 1 - \chi\big( \frac{4r}{r'} \big) - \chi\big( \frac{4r'}{r} \big) \Big).\]
Then we have 
\[PG_{nd}=\Id+E_{nd},\]
where the error term $E_{nd}$ is smooth across the diagonal, and vanishes to first order at $\zf$ (as a $b$-half density) and to infinite order at all other boundary hypersurfaces.  We may assume that $G_{nd}$ is supported close to the union of the diagonal, $\zf$, and $\sf$. 

We now define our global parametrix to be 
\begin{equation}
G_a = G_f + G_{nd}.
\label{para}\end{equation}

\subsection{The indicial operator at $\zf$}
In this subsection we show that the leading behaviour of $G_a$ at $\zf$ agrees with that of $G_{nd}$. To do this, it suffices to show that  the indicial operator of $\tilde G$ agrees (at least for $r/r' < 1/4$ and $r/r' > 4$, where we have shown convergence of the series) with that of $\tilde G_{nd}^{zf}$. By Proposition~\ref{compacterror}, the indicial operator of $\tilde G_{nd}^{zf}$ is equal to $I_b(\tilde P)^{-1}$. Let us now determine this indicial operator 

The indicial operator of $\tilde{P}$ is
\[\ib\big(\tilde{P}\chi(r)\big)=-(r\partial_r)^2+\Delta_Y+V_0(y)+\Big(\frac{d-2}{2}\Big)^2.\]
Let $\mu_j^2, u_j, \Pi_j$ be the same as defined in Section \ref{formula}. Here, instead of (\ref{ptilde}) and (\ref{ttilde}) we have
\[\ib\big(\tilde{P}\chi(r)\big)=\sum_j\Pi_j S_j,\]
where
\begin{equation*}
S_j=-(r\partial_r)^2+\mu_j^2.
\end{equation*} 
Similar to Section \ref{formula}, then the kernel $S_j^{-1}$ is
\[S^{-1}_j(r,r')=
\begin{cases}
\frac{1}{2\mu_j}(\frac{r}{r'})^{\mu_j}\big|\frac{dr}{r}\frac{dr'}{r'}\big|^{\frac{1}{2}},\hspace{3mm}r<r',\\
\frac{1}{2\mu_j}(\frac{r'}{r})^{\mu_j}\big|\frac{dr}{r}\frac{dr'}{r'}\big|^{\frac{1}{2}},\hspace{3mm}r>r'.
\end{cases}\]
 Hence
\begin{equation}\label{ibgbsy}
\big(\ib(\tilde P)\big)^{-1}(s, y, y')=\begin{cases}
\frac{1}{2}\sum_j \frac{1}{\mu_j}u_j(y)\overline{u_j(y')}s^{+\mu_j}\big|\frac{ds}{s}dhdh'\big|^{\frac{1}{2}},\hspace{3mm}s>1,\\
\frac{1}{2}\sum_j \frac{1}{\mu_j}u_j(y)\overline{u_j(y')}s^{-\mu_j}\big|\frac{ds}{s}dhdh'\big|^{\frac{1}{2}},\hspace{3mm}s<1,
\end{cases} \quad s = \frac{r}{r'}. 
\end{equation}
The convergence of this sum can be analyzed using Lemma~\ref{sumlemma}: the sum converges smoothly for $s < 1$ and for $s >1$.

Now we determine the leading behaviour of $\tilde G$ at $\zf$. We only consider the case $\frac{r}{r'}\ < 1/4$, as the case $\frac{r}{r'} > 4$ is completely parallel.
Recall from expression (\ref{infiniteseries}), for $\frac{r}{r'} < 1/4$ we have, 
\[\tilde{G}(r, r', y, y')=\sum_j u_j(y)\overline{u_j(y')}I_{\mu_j}(r)K_{\mu_j}(r')  \bigg|\frac{dr}{r}\frac{dr'}{r'}dhdh'\bigg|^{\frac{1}{2}}
  .\]
We use the limiting forms for small arguments from \cite[Sec. 9.6]{AS}, that is when $r$, $r'\rightarrow 0$, 
\begin{equation}\label{Iexpansion}
I_{\mu_j}(r)=\frac{r^{\mu_j}}{2^{\mu_j}\Gamma(\mu_j+1)}+O(r^{\mu_j+2}),
\end{equation}
and
\begin{equation}\label{Kexpansion}
K_{\mu_j}(r')=\frac{2^{\mu_j-1}\Gamma(\mu_j)}{r'^{\mu_j}}+O({r'}^{-\mu_j+2} |\log r'|).
\end{equation}
Therefore, 
\begin{equation}\label{compatibility}
\tilde{G}_0(r, r', y, y')
=\frac{1}{2}\sum_j \frac{1}{\mu_j}u_j(y)\overline{u_j(y')}\Big(\frac{r}{r'}\Big)^{\mu_j}\bigg|\frac{dr}{r}\frac{dr'}{r'}dhdh'\bigg|^{\frac{1}{2}}
+O(r'^2 |\log r'|).
\end{equation}
Since $\frac{dr}{r}\frac{dr'}{r'}=\frac{ds}{s}\frac{dr'}{r'}$, expression (\ref{compatibility}) of $\tilde{G}(r, r', y, y')$ at $\zf$
is indeed consistent with 
expression (\ref{ibgbsy}) of $(\ib(\tilde P))^{-1}(s, y, y') = \ib(\tilde G_{nd}^{zf})$ when restricted to the $zf$-face. 

Finally, since the cutoff function used to define $G_f$ is $\chi(4r/r') + \chi(4r'/r)$, and that used to define $G_{nd}$ is $1 - \chi(4r/r') - \chi(4r'/r)$, we see that $G_a$ has the same leading asymptotic at $\zf$ as $\tilde G_{nd}^{zf}$, namely $I_b(\tilde P)^{-1}$.

\subsection{Construction of $P^{-1}$}
We have constructed an approximate inverse $G_a = G_f + G_{nd}$; let $E$ be the corresponding error term:
\[PG_a=\Id+E.\]
We next try to  solve away $E$ to obtain our final $G=P^{-1}$. 
We begin by summarising the properties of $G_a$ and $E$. 

\begin{proposition}\label{Gaproperties}
As a multiple of the Riemannian half-density $|r^{d-1}r'^{d-1}drdr'dh dh'|^{\frac{1}{2}}$ on the blown-up space,
the kernel $G_a$ is the sum of two terms. 
One is $G_{nd}$, supported where $1/8 \leq r/r' \leq 8$, 
and is such that 
$\rho_{\zf}^{d-2}G_{nd}$ is conormal of order $-2$ with respect to the diagonal uniformly up to both $\zf$ and $\sf$, 
where $\rho_{\zf}$ is any boundary defining function for $\zf$, and is rapidly decreasing at $\bfc$. 
The other term $G_f = G_a-G_{nd}$ satisfies:
\begin{enumerate}
\item[\rm (i)] it is smooth at the diagonal, and polyhomogeneous conormal at all boundary hypersurfaces;
\item[\rm (ii)] it vanishes to infinite order at $\lbi${}, $\rbi${} and $\bfc$;
\item[\rm (iii)] it vanishes to order $1-\frac{d}{2}+\mu_0$ at $\lbz${} and $\rbz${};
\item[\rm (iv)] it vanishes to order $2-d$ at $\zf$.
\end{enumerate}
\end{proposition} 
\begin{proof}
The properties of $G_{nd}$ follow from properties of the full $b$-calculus and of the scattering calculus recalled in Section~\ref{sec:bsc}. 

The diagonal part of property (i) of $G_f$ is clear: in fact it is supported away from the diagonal. Polyhomogeneity of $G_f$ at $\lbz$ and $\rbz$ follows from Proposition~\ref{prop:phg} and the symmetry of $G_f$, while polyhomogeneity (in a trivial sense, with an empty index set) at $\lbi$, $\rbi$, and $\bfc$ follows from the exponential decrease of $G_f$ as $r$ or $r'$ tend to infinity, as shown by Lemma~\ref{sumlemma}. 

We obtain the vanishing order at $\lbz${} from equations (\ref{initial}) and (\ref{GG0tilde}). 
Since $r$ is the boundary defining function for $\lbz${}, we need to work out its power. 
Clearly one power of $r$ comes from (\ref{GG0tilde}), 
while $I_{\mu_0}(r)$ in (\ref{initial}) gives us the power $r^{\mu_0}$.  Then the difference between the $b$-half density and the Riemannian half-density gives us a power of $r^{-d/2}$ (as in \eqref{KKtilde}). 
Combining these we conclude that the vanishing order at $\lbz${} is $1-\frac{d}{2}+\mu_0$. 
The vanishing order at $\rbz${} is similar.

Last, we show (iv). Since both $r$ and $r'$ vanish at $zf$, to obtain the vanishing order of $G_a$
at $zf$, as a scattering-half-density, we combine the powers of $r$ and $r'$ in (\ref{Gzfdefinition}) and \eqref{GG0tilde} with the factor $(rr')^{-d/2}$ involved in the change from a $b$-half density to the Riemannian half-density.  
So the order of vanishing is $1-\frac{d}{2}+1-\frac{d}{2}=2-d$.
\end{proof}

\begin{proposition}\label{Eproperties}
The error term $E$ has the following properties on the blown-up space:
\begin{enumerate}
\item[\rm (i)] it is smooth in the interior;
\item[\rm (ii)] it vanishes to the first order (as a $b$-half-density, or to order $1-d$ as a Riemannian half-density) at the $zf$-face;
\item[\rm (iii)] it vanishes to infinite order at $\lbz$, $\rbz$, $\lbi$, $\rbi$, $\sf$ and $\bfc$;
\item[\rm (iv)] it is compact on $L^2(M)$; in fact its Schwartz kernel is Hilbert-Schmidt. 
\end{enumerate}
Moreover, the $k$-fold composition $E^k$ satisfies similar conditions, with (ii) strengthened to vanishing to order $k$ at $\zf$ as a $b$-half-density. 
\end{proposition}
\begin{proof}
Property (i) follows from the choice of $G_{nd}$. Property (ii) follows from the fact that the indicial operator of $(rr')^{-1} G_a$ is equal to $I_b(\tilde P)^{-1}$, as shown in the previous subsection.
Property (iii) follows from the fact that elements of the scattering calculus vanish to infinite order at $\bfc$, the fact that $G_f$ is equal to the exact inverse of $P$ outside the region $\{ 1/8 \leq r/r' \leq 8 \}$ (so in fact $E$ is supported in this region, hence vanishes in a neighbourhood of $\lbz$, $\rbz$, $\lbi$ and $\rbi$), and the 
exponential vanishing of $G_f$ as either $r \to \infty$ or $r' \to \infty$ --- see \eqref{expvanishing}.  Properties (i), (ii) and  (iii) show that $E$ has an $L^2$ kernel, proving  Property (iv). 

To show the last remark, we use a smooth cutoff function to divide $E$ into two parts, $E = E_b + E_{sc}$, where $E_b$ is an order $-\infty$ operator in the $b$-calculus, vanishing to first order at $\zf$, and $E_{sc}$  is an order $(-\infty , \infty)$ operator in the scattering calculus. Then $E^k = (E_b + E_{sc})^k$. Any mixed terms will vanish to infinite order at each boundary hypersurface. Of the remaining terms, using the composition properties of the $b$- and scattering calculus recalled in Section~\ref{sec:bsc}, $E_b^k$ is order $-\infty$ in the $b$-calculus and vanishes to order $k$ at $\zf$, while $E_{sc}^k$ is order $(-\infty, \infty)$ in the scattering calculus. Moreover, $E^k$ is supported where $\{ 8^{-k} \leq r/r' \leq 8^k \}$, hence vanishes in a neighbourhood of $\lbz$, $\rbz$, $\lbi$ and $\rbi$. 
\end{proof}

We proceed to solve away $E$. 
To achieve that, we would like to invert $\Id+E$. But it might not be invertible: if not,  we perturb $G_a$ so that $\Id+E$ becomes invertible.

Since $E$ is compact on $L^2(M)$, according to Proposition~\ref{Eproperties},  
$\Id+E$ is Fredholm of index $0$, and its null space and the complement of its range both have the same finite dimension, say $N$. 
Removing the null space gives us an invertible operator, and to achieve that we add a rank $N$ operator to $G_a$.
To construct the rank $N$ operator we need the following lemma. 
\begin{lemma}\label{psiphi}
There exist smooth functions $\psi_1, ..., \psi_N$, $\phi_1, ..., \phi_N$ on $M$ such that 
\begin{enumerate}
\item[\rm(i)] $\psi_1, ..., \psi_N$ span the null space of $\Id+E$, and $P\phi_1, ..., P\phi_N$ span a space supplementary to the  range of $\Id+E$;
\item[\rm(ii)] they are $O(r^\infty)$ as $r\rightarrow 0$ and $O(r^{-\infty})$ as $r\rightarrow\infty$.
\end{enumerate}
\end{lemma}
\begin{proof}
We choose the $\psi_i$ to be any basis of the null space of $\Id+E$. To obtain property (ii) for the $\psi_i$, we note that $\psi_i=-E(\psi_i)$, hence iterating, we have $\psi_i =  E^{2N} \psi_i$ for each $N \geq 1$. Now we consider mapping properties of the operator $E^N$. First, writing $E = E_b + E_{sc}$ as in the proof of Proposition~\ref{Eproperties}, it is easy to see that $E_{sc}$ and $\nabla E_{sc}$ map $L^2(M)$ to $\jap{r}^{-K} L^2(M)$ for arbitrary $K$. (Here $\nabla$ is shorthand for the vector of derivatives $(\partial_r, r^{-1} \partial_{y_i})$.) As for $E_b$, since it has negative order in the $b$-calculus and vanishes to first order at $\zf$, we see that $E_b$ maps $L^2(M)$ to $r L^2(M)$. Since the kernel $(r/r')^a E$ has the same properties as $E$ listed in Proposition~\ref{Eproperties}, it follows that $E_b$ maps $r^a L^2(M)$ to $r^{a+1} L^2(M)$ for any $a$. Also, applying a derivative $\nabla = (\partial_r, r^{-1} \partial_{y_i})$ to $E_b$, it is still of negative order in the $b$-calculus, though no longer vanishing at $\zf$, so we see that $\nabla E_b$ maps $r^a L^2(M)$ to $r^{a} L^2(M)$ for any $a$. Summarizing, we have
\begin{equation}\begin{gathered}
E \text{ boundedly maps } r^a L^2(M) \to r^{a+1} \jap{r}^{-K} L^2(M), \\
\nabla E \text{ boundedly maps } r^a L^2(M) \to r^{a} \jap{r}^{-K} L^2(M).
\end{gathered}\label{Emapping}\end{equation}
Applying these properties of $E$ iteratively, we see that $E^{2N}$ maps $L^2(M)$ to $r^{N} \jap{r}^{-2N} H^N(M)$ for any $N$. Hence, using Sobolev embeddings, $\psi$ is smooth and has rapid decay both as $r \to 0$ and $r \to \infty$. 

As for the $\phi_i$, to show that we can choose functions $\phi_1, \dots, \phi_N$ as above, it is sufficient to show that the range of $P$ on the subspace $\mathcal{S}$ of  smooth half-densities satisfying (ii) is dense on $L^2(M)$. If this were not true, then there would be a nonzero half-density $f \in L^2(M)$ orthogonal to the range of $P$ on such half-densities: that is, we would have
$$
\langle P u, f \rangle = 0, \text{ for all } u \in \mathcal{S}. 
$$
Since $\mathcal{S}$ is a dense subspace, this implies that $Pf = 0$ distributionally. By elliptic regularity this means that $f$ is smooth and $Pf = 0$ strongly, but since $P$ is invertible on $L^2$ this implies $f = 0$, a contradiction. Therefore we can choose the $\phi_i \in \mathcal{S}$ as desired. 
\end{proof}

Let $Q$ be the rank   $N$ operator  
\[Q=\sum_{i=1}^N \phi_i\langle\psi_i,\cdot \rangle,\]
where $\langle\psi_i,\cdot\rangle$ means the inner product with $\psi_i$. 
The functions $\psi_1,...,\psi_N$, 
$\phi_1,...\phi_N$ are chosen as in Lemma \ref{psiphi}.
Then we have 
\[P(G_a+Q)=\Id+E+PQ,\]
which is invertible. From here we obtain
\[P^{-1}=(G_a+Q)(\Id+E+PQ)^{-1}.\]
Using property (ii) of Lemma~\ref{psiphi}, we see that $G_a+Q$ has the `same' properties as $G_a$, ie it has those properties listed in Proposition \ref{Gaproperties}, and $E' := E + PQ$ has  properties (i) -- (iv) listed in Proposition \ref{Eproperties}.
Define operator $S$ by 
\[S=(\Id+E')^{-1}-\Id.\]
Then we can write
\[P^{-1}=(G_a+Q)(\Id+S).\]
We need to know the properties of $S$. 

\begin{lemma}\label{Sproperties}
The operator $S$ has properties (i) -- (iv) listed in Proposition \ref{Eproperties}.
\end{lemma}

\begin{remark}
A similar analysis was done in \cite[Sec. 5.4]{GHS}.
\end{remark}

\begin{proof}
Using the identities  $(\Id+S)(\Id+E')=(\Id+E')(\Id+S)=\Id$, we obtain 
\begin{equation}\label{Sequation}
S=-E'+E'^2+E'SE'.
\end{equation}
For any positive integer $N$, we substitute the expression (\ref{Sequation}) into itself $2N-1$ times, and we get
\begin{equation}\label{SNtimes}
S=\sum_{j=1}^{4N}(-1)^jE'^j+E'^{2N}SE'^{2N}.
\end{equation}
Using the last part of Proposition~\ref{Eproperties}, we see that the  term $\sum_{j=1}^{4N}(-1)^jE'^j$  has all the properties listed in the Lemma, 
so we focus on the term $S_N := E'^{2N}SE'^{2N}$. 
Using  \eqref{Emapping}, we see that $r^{-N} \jap{r}^{2N} \nabla_z^{(N)} {E'}^{2N}$ and ${r'}^{-N} \jap{r'}^{2N} \nabla_{z'}^{(N)} {E'}^{2N}$ are bounded operators on $L^2$. Since $S$ is Hilbert-Schmidt, it follows that
$(r r')^{-N} \jap{r}^{2N} \jap{r'}^{2N} \nabla_z^{(N)} \nabla_{z'}^{(N)}  S_N$ has an $L^2$ kernel. Using Sobolev embeddings, this gives regularity and vanishing (at the boundary) of $S_N$ of some finite order $N + O(1)$, and hence the same finite order regularity and vanishing of $S$. Since this argument can be made for any $N$, this proves that $S$ has the properties (i) -- (iv) listed in Proposition \ref{Eproperties}.
\end{proof}

To summarise, we have 
\[G=P^{-1}=(G_a+Q)(\Id+S),\]
where $G_a+Q$ has those properties listed in Proposition \ref{Gaproperties}, 
$\Id+S$ is a compact operator, 
and $S$ has those properties listed in Lemma \ref{Sproperties}.
Our final step is to analyze the composition $(G_a+Q)(\Id+S)$ and show that $G$ itself satisfies all the conditions listed in Proposition~\ref{Gaproperties}. 
We summarise key information about $G=P^{-1}$ obtained through our construction in the following theorem. To state it, define 
$\omega = 1 - \chi(4r/r') - \chi(4r'/r)$ where $\chi$ is as in \eqref{initial}; thus, $\omega$ is a smooth function on the blown up space supported away from $\lbz, \lbi, \rbz$ and $\rbi$, and equal to $1$ on a neighbourhood of the diagonal.  Also let $\rho_{\zf}$ be a boundary defining function for $\zf$.
\begin{theorem}\label{Gproperties}
Let $G_c = \omega G$ and $G_s = (1 - \omega) G$. Then, as  a multiple of the Riemannian half-density, ie the scattering-half-density $|r^{d-1}r'^{d-1}drdr'dh dh'|^{\frac{1}{2}}$, 
on the blown-up space, 
$\rho_{\zf}^{d-2}G_c$ is conormal of order $-2$ with respect to the diagonal uniformly 
up to both $zf$ and $\sf$, while  $G_s$ satisfies properties (i)-(iv) listed in Proposition \ref{Gaproperties}.
\end{theorem}

\begin{remark}
The subscripts $c$ and $s$ are chosen to indicate that 
$G_c$ is the part of $G$ which is conormal at the diagonal, 
while $G_s$ is the part of $G$ which is smooth at the diagonal. 
\end{remark}

\begin{proof}
We have already proved these properties for $G_a$, in Proposition~\ref{Gaproperties}, so we need to check them for the terms $Q+QS+G_aS = G - G_a$. 
Since $Q$ and $QS$ both are smooth and vanish to infinite order at the boundaries, these terms trivially satisfy all the conditions. So 
it remains to check that $G_aS$ has the same properties as $G_a$.

We write $G_aS$ as a sum of two parts. Let $\eta : [0, \infty) \to [0, 1]$ be a smooth cutoff function such that $\eta([0,1]) = 0$ and $\eta([2, \infty)) = 1$. 
The first part $\eta(r)G_a\eta(r')$ is in the scattering calculus. 
Note that $\eta(2r')S$ is also in the scattering calculus, and that 
$$
\big(\eta(r)G_a\eta(r') \big) \big(\eta(2r')S \big) = \big(\eta(r)G_a\eta(r') \big) S.
$$
Therefore by Proposition \ref{scatteringclosure}, this term 
is in the scattering calculus.
The second part $G_a-\eta(r)G_a\eta(r')$ is in the 
full $b$-calculus. (Although the support of this term meets the boundary hypersurfaces $\lbi$ and $\rbi$, its Schwartz kernel is rapidly vanishing there, enabling us to regard it as living in the $b$-calculus.)  In a similar sense, $S$ is in the small $b$-calculus (it vanishes rapidly at every boundary hypersurface except $\zf$). 
Therefore by Proposition \ref{bBcomposition}, 
$\big(G_a-\eta(r)G_a\eta(r')\big)S$ is in the full $b$-calculus, with the same index sets at $\lbz$ and $\rbz$ as $G_a$. 
Therefore the required properties for $G_c$ and property (i) for $G_s$ follow. Also, since $S$ vanishes to first order at $\zf$, the same is true for the composition $\big(G_a-\eta(r)G_a\eta(r')\big)S$. So $G_a S$ has the same vanishing orders (or better) at the boundary hypersurfaces as $G_a$. 
\end{proof}

\begin{remark}
In the case of the potential $V\equiv 0$, we have $\mu_0=\frac{d}{2}-1$. 
So the vanishing order in item (iii) of Theorem \ref{Gproperties} becomes $0$. 
It's consistent with the case when the cone is $\mathbb{R}^d$ and the potential $V\equiv 0$, 
when the cone tip can be chosen arbitrarily, and $G$ is smooth everywhere.
\end{remark}

The vanishing orders of $G=P^{-1}$ at various boundaries of the blown-up space are shown in Figure \ref{vanish}.

\begin{figure}[h!]
\centering
\includegraphics[scale=0.35]{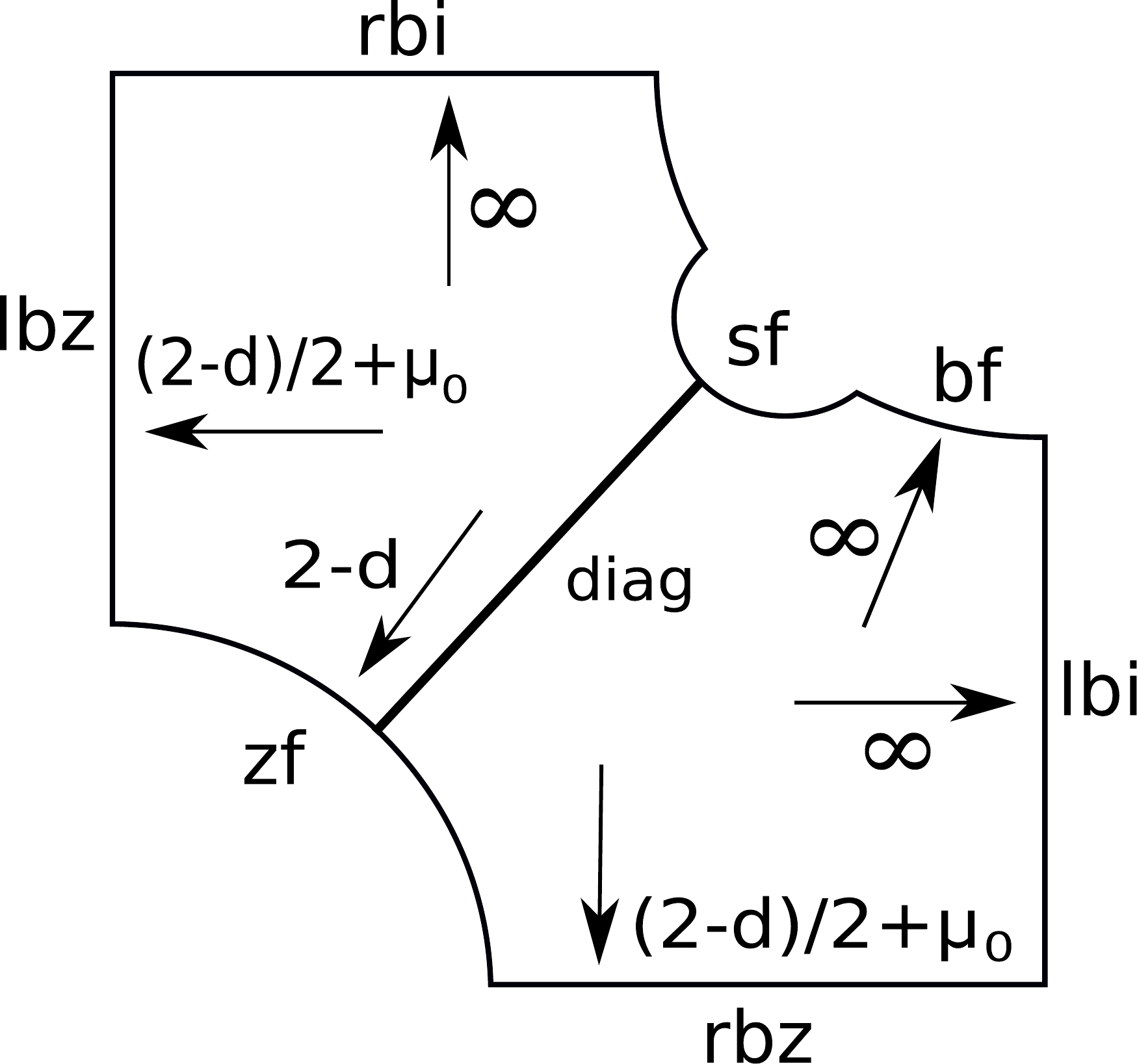}
\caption{The vanishing properties at various boundaries}
\label{vanish}
\end{figure}

\begin{remark}
The construction of $G=P^{-1}$ in this section is sketched in the paper \cite{GH} by C. Guillarmou and the first author, but details are lacking. 
It's not fully justified in \cite{GH} that the kernel is in the scattering calculus near $\sf$, and in the 
$b$-calculus near $\zf$. For this reason we have given complete details in this section.
\end{remark}


\section{The Boundedness of the Riesz transform}
\label{chapter6}

\subsection{Estimate on the kernel}

Recall that the Riesz transform $T$ with the inverse square potential $V=\frac{V_0}{r^2}$, defined in Section~\ref{chapter5}, can be expressed as
\[T=\frac{2}{\pi}
\int_0^\infty\nabla\big(H+\lambda^2\big)^{-1}d\lambda,\]
where $H$ is given by \eqref{Hdefn}, 
and recall that $H$ is homogenous of degree $-2$.
Our analysis of the Riesz transform will be based on the following estimate on the kernel $T(z, z')$. 
\begin{proposition}\label{esimateonT}
We have the following estimate on the kernel of $T$,
\begin{equation*}
\begin{split}
|T(z,z')|\lesssim\int_0^\infty\lambda^{d-2}\big|\nabla\big(G(\lambda z, \lambda z')\big)\big|d\lambda,
\end{split}
\end{equation*}
where $G=P^{-1} = (H+1)^{-1}$, with properties listed in Theorem \ref{Gproperties}.
\end{proposition}
\begin{proof}
This comes from the relationship between $(H+\lambda^2)^{-1}$ and $G=(H+1)^{-1}$, which is 
\[(H+\lambda^2)^{-1}(z, z')=\lambda^{d-2}(H+1)^{-1}(\lambda z, \lambda z').\]
The power $-2$ of $\lambda$ appears because $H$ is homogenous of degree $-2$. 
Remember these kernels are Riemannian half-densities, and this accounts for the power $d$ of $\lambda$:
\begin{equation*}
\begin{split}
|(\lambda r)^{d-1}(\lambda r')^{d-1}d(\lambda r)d(\lambda r')dhdh'|^{\frac{1}{2}}
=\lambda^d|r^{d-1}r'^{d-1}drdr'dhdh'|^{\frac{1}{2}}.
\end{split}
\end{equation*}
\end{proof}

\subsection{Boundedness on $L^2(M)$}

\begin{proposition}\label{bon2}
The Riesz transform $T$ with the inverse square potential $V=\frac{V_0}{r^2}$ is bounded on $L^2(M)$.
\end{proposition}
\begin{proof}
Our assumption is $\Delta_Y+V_0(y)+(\frac{d-2}{2})^2>0$, ie $\Delta+\frac{1}{r^2}V_0(y)>0$. 
Hence there is $\varepsilon>0$ such that  $\Delta+\frac{1}{(1-\varepsilon)r^2}V_0(y)>0$. 
That is, $\Delta+\frac{1}{r^2}V_0(y)>\varepsilon\Delta$. 
From here,
\begin{equation*}
\begin{split}
\langle Tf, Tf \rangle &=\langle \Delta\big(\Delta+\frac{1}{r^2}V_0(y)\big)^{-\frac{1}{2}}f, \big(\Delta+\frac{1}{r^2}V_0(y)\big)^{-\frac{1}{2}}f \rangle\\
&\leq\langle \varepsilon^{-1}\big(\Delta+\frac{1}{r^2}V_0(y)\big)\big(\Delta+\frac{1}{r^2}V_0(y)\big)^{-\frac{1}{2}}f, \big(\Delta+\frac{1}{r^2}V_0(y)\big)^{-\frac{1}{2}}f \rangle\\
&=\varepsilon^{-1}\langle \big(\Delta+\frac{1}{r^2}V_0(y)\big)^{\frac{1}{2}}f, \big(\Delta+\frac{1}{r^2}V_0(y)\big)^{-\frac{1}{2}}f \rangle\\
&=\varepsilon^{-1}\langle f, f \rangle.
\end{split}
\end{equation*}
Therefore $T$ is bounded on $L^2(M)$.
\end{proof}

\subsection{The diagonal region}
To understand the Riesz transform on $L^p$, we break up $G$ as in Theorem~\ref{Gproperties}. Here we will write $G_1$ for $G_c = \omega G$ (recall $\omega = 1 - \chi(4r/r') - \chi(4r'/r)$),  and we further decompose $G_s = G_2 + G_3$, where $G_2 = G \chi(4r/r')$ and $G_3 = G \chi(4r'/r)$. Notice that $G_2$ and $G_3$ are supported away from the diagonal, in particular where the infinite series \eqref{infiniteseries} has good convergence properties as shown in Proposition~\ref{prop:phg}. 
We correspondingly break up the Riesz transform into three pieces. Thus we have 

\begin{equation}
T_i(z,z')
= \frac{2}{\pi} \int_0^\infty\lambda^{d-2}\nabla_z\big(G_i(\lambda z, \lambda z')\big) \, d\lambda.\\
\label{Tidefn}\end{equation}

We now show that $T_1$ is of weak type $(1, 1)$. 
For that we first need to estimate the derivatives of $G_1$. 

\begin{lemma}\label{gdestimate}
Let $\dist(z, z')$ denote the distance between $z$ and $z'$ on $M$. On the support of $\omega$,  we have $\rho_{\bfc}\lesssim\dist(z, z')^{-1}$, 
where $\rho_{\bfc}$ is a boundary defining function for $\bfc$.
\end{lemma}
\begin{proof}
Let $z=(r, y)$ and $z'=(r', y')$. 
Observe from \eqref{distanceonm} that $\dist(z, z')$ is bounded above by $r+r'$. 
Therefore in the region $\{ 1/8 \leq r/r' \leq 8\}$ we have
\[\dist(z, z')^{-1}\geq (r+r')^{-1}=r'^{-1}(1+\frac{r}{r'})^{-1}\geq\frac{1}{9}r'^{-1}.\]
As $r'^{-1}$ is a boundary defining function for $\bfc$ on the support of $\omega$, the result follows. 
\end{proof}

\begin{lemma}\label{corofort1}
The kernel $G_1$ satisfies the estimate that for any integer $j\geq 0$, we have
\[|\nabla_{z,z'}^jG_1(z, z')|\lesssim
\begin{cases}
\dist(z, z')^{2-d-j},\hspace{3mm}&\dist(z, z')\leq 1,\\
\dist(z,z')^{-N},&\dist(z, z')\geq 1,
\end{cases}\]
for any $N>0$.
\end{lemma}
\begin{proof}
Note that $G_1$ is supported in the region $R_1 = \{ 1/8 \leq r/r' \leq 8\}$. 
Since $\rho_{\zf}^{d-2} G_1$ is conormal of order $-2$ with respect to the diagonal, by  Proposition \ref{quadraticdefining},
near the diagonal we have 
\[ \big| \rho_{\zf}^{d-2}G_1(z, z') \big|\lesssim a_{\rm diag}^{\frac{2-d}{2}}.\]
Near $\zf$, $a_{\rm diag}=\frac{\dist(z, z')^2}{\rho_{\zf}^2}$, so it follows that near $\zf$, we have 
\[ \big|G_1(z, z') \big|\lesssim \dist(z, z')^{2-d}.\]
Away from $zf$, $a_{\rm diag}=\dist(z, z')^2$, therefore
\[ \big|G_1(z, z') \big|\lesssim a_{\rm diag}^{\frac{2-d}{2}}=\dist(z, z')^{2-d}.\]
Now let's consider the behaviour of $G_1$ near $\bfc$. 
By Theorem \ref{Gproperties}, we know that it vanishes to infinite order at $\bfc$, while by Lemma \ref{gdestimate}, we know that 
$\rho_{\bfc}\lesssim\dist(z, z')^{-1}$. Therefore near the $\bfc$-face we know that
\[|G_1(z, z')|\lesssim\dist(z, z')^{-N},\]
for any $N>0$.
The rest of $R_1$ is easy because after we take away 
the neighbourhoods near $\zf$, $\bfc$ and the diagonal, we are left with a compact set,
on which both $G_1$ and $\dist(z, z')^{-1}$ are continuous with $\dist(z, z')^{-1}$ being non-zero. 
Therefore we can conclude that 
\[|G_1(z, z')|\lesssim
\begin{cases}
\dist(z, z')^{2-d},\hspace{3mm}&\dist(z, z')\leq 1,\\
\dist(z,z')^{-N},&\dist(z, z')\geq 1,
\end{cases}\]
for any $N>0$.
Then using the conormality of $G$ at the diagonal and 
polyhomogeneous conormality of $G$ at the boundary hypersurfaces, 
we obtain the estimates on $\nabla_{z,z'}^jG_1(z, z')$.
\end{proof}

\begin{proposition}\label{T1}
The operator $T_1$ maps $L^1(M)$ into $L^{1, weak}(M)$.
\end{proposition}
\begin{proof}
We just apply the Calder{\' o}n-Zygmund theory, see \cite[Section 8.1.1]{LG}. It is sufficient to verify the following conditions:
\begin{enumerate}
\item[\rm (i)] $T_1$ is bounded on $L^2(M)$;
\item[\rm(ii)] $|T_1(z,z')|\leq\frac{C}{\big(\dist(z,z')\big)^d}$;
\item[\rm(iii)] $|\nabla_zT_1(z,z')|\leq\frac{C}{\big(\dist(z,z')\big)^{d+1}}$ and $|\nabla_{z'}T_1(z,z')|\leq\frac{C}{\big(\dist(z,z')\big)^{d+1}}$,
\end{enumerate}
for some constant $C>0$. 

We already know from Proposition \ref{bon2} that $T$ is bounded on $L^2(M)$.
So to verify condition (i), we just need to show $T-T_1$ is bounded on $L^2(M)$, 
which is covered by Proposition \ref{T2T3} in Section \ref{sectionr2r3}.

Now we show conditions (ii) and (iii). 
By Lemma \ref{corofort1} we know the kernel $G_1$ satisfies, 
with any $\lambda>0$, 
\[\big|\nabla_z\big(G_1(\lambda z, \lambda z')\big)\big|
\leq\lambda\big|(\nabla_zG_1)(\lambda z, \lambda z')\big|
\lesssim
\begin{cases}
\lambda^{2-d}\dist(z, z')^{1-d},\hspace{3mm}&\lambda\dist(z, z')\leq 1,\\
\lambda^{-N+1}\dist(z, z')^{-N},&\lambda\dist(z, z')\geq 1,
\end{cases}\]
and
\[\big|\nabla_z^2\big(G_1(\lambda z, \lambda z')\big)\big|\leq
\lambda^2\big|(\nabla_z^2G_1)(\lambda z, \lambda z')\big|\lesssim
\begin{cases}
\lambda^{2-d}\dist(z, z')^{-d},\hspace{3mm}&\lambda\dist(z, z')\leq 1,\\
\lambda^{-N+2}\dist(z, z')^{-N},&\lambda\dist(z, z')\geq 1,
\end{cases}\]
for any $N>0$. 
We use this to estimate $T_1(z, z')$, 
\begin{equation*}
\begin{split}
|T_1(z,z')|
&\lesssim\int_0^\infty\lambda^{d-2}\big|\nabla_z\big(G_1(\lambda z, \lambda z')\big)\big|d\lambda\\
&\lesssim\int_0^{\frac{1}{\dist(z, z')}}\dist(z, z')^{1-d}d\lambda+\int_{\frac{1}{\dist(z, z')}}^\infty\lambda^{d-N-1}\dist(z, z')^{-N}d\lambda\\
&=\dist(z, z')^{-d}+\dist(z, z')^{-N}\int_{\frac{1}{\dist(z, z')}}^\infty\lambda^{d-N-1}d\lambda\\
&=\dist(z, z')^{-d}+\dist(z, z')^{-d}\hspace{13mm}(\mbox{Choose }N=d+1.)\\
&=\frac{2}{\dist(z, z)^d}.
\end{split}
\end{equation*}
Now estimate the derivative with respect to $z$. The $z'$ case is similar.
\begin{equation*}
\begin{split}
|\nabla_z T_1(z,z')|
&=\frac{2}{\pi}\bigg|\int_0^\infty\lambda^{d-2}\nabla_z\nabla_z\big(G_1(\lambda z, \lambda z')\big)d\lambda\bigg|\\
&\lesssim\int_0^\infty\lambda^{d-2}\big|\nabla_z^2\big(G_1(\lambda z, \lambda z')\big)\big|d\lambda\\
&\lesssim\int_0^{\frac{1}{\dist(z, z')}}\dist(z, z')^{-d}d\lambda+\int_{\frac{1}{\dist(z, z')}}^\infty\lambda^{d-N}\dist(z, z')^{-N}d\lambda\\
&=\dist(z, z')^{-d-1}+\dist(z, z')^{-N}\int_{\frac{1}{\dist(z, z')}}^\infty\lambda^{d-N}d\lambda\\
&=\dist(z, z')^{-d-1}+\dist(z, z')^{-d-1}\hspace{13mm}(\mbox{Choose }N=d+2.)\\
&=\frac{2}{\dist(z, z')^{d+1}}.
\end{split}
\end{equation*}
This completes the proof.
\end{proof}

By interpolation, we obtain the following proposition. 

\begin{proposition}\label{interpolation}
The operator $T_1$ is bounded on $L^p(M)$ for any $p>1$.
\end{proposition}
\begin{proof}
By Marcinkiewicz Interpolation Theorem, we know that $T_1$ is bounded on $L^p(M)$ for all $1<p\leq 2$. The same holds for the adjoint of $T_1$. Using duality, we get boundedness for $1 < p < \infty$. 
\end{proof}

\subsection{Off-diagonal region}\label{sectionr2r3}
To study the boundedness of the two off-diagonal operators $T_2$ and $T_3$, the following lemmas will be useful. 
They are similar to \cite[Lemma 5.4]{HS} but not covered by it.

\begin{lemma}\label{lemmaupper}
Consider the kernel $K(r, r')$ defined by
\[K(r, r')=
\begin{cases}
r^{-\alpha}r'^{-\beta},\hspace{5mm}&r\leq r',\\
0,& r> r'.\\
\end{cases}\]
If $\alpha+\beta=d$, $\beta>0$, and $p$ satisfies
\begin{equation}
p<\frac{d}{\max(\alpha, 0)},
\end{equation}
then $K$ is bounded as an operator on $L^p(\mathbb{R}_+; r^{d-1}dr)$.
\end{lemma}
\begin{proof}
The proof is essentially taken from \cite{HS}. 
To find out for what $p$ the operator with kernel $K(r, r')$ is bounded on $L^p(\mathbb{R}_+, r^{d-1}dr)$,
we consider the isometry $M: L^p(\mathbb{R}_+, r^{d-1}dr)\rightarrow L^p(\mathbb{R}_+, r^{-1}dr)$ defined by
\[(Mf)(r)=r^{\frac{d}{p}}f(r).\]
Then the kernel of the operator $\tilde{K}=MKM^{-1}: L^p(\mathbb{R}_+, r^{-1}dr)\rightarrow L^p(\mathbb{R}_+, r^{-1}dr)$ is 
\[\tilde{K}(r, r')=r^{\frac{d}{p}}r'^{d-\frac{d}{p}}K(r, r')=(\frac{r}{r'})^{-\alpha+\frac{d}{p}}\chi_{\{ r \leq r'\}}.\]
Perform a substitution $s=\ln r$, $s'=\ln r'$, then $\tilde{K}(s, s')$ is an operator on $L^p(\mathbb{R}, ds)$, and
\[\tilde{K}(s, s')=e^{(-\alpha+\frac{d}{p})(s-s')}\chi_{\{s-s'\leq 0\}}.\]
This is a convolution operator, so it is bounded provided the kernel is an $L^1$-function with variable $s-s'$. 
Since $s-s'\leq 0$, we want $-\alpha+\frac{d}{p}>0$. That is, 
\[p<\frac{d}{\max(\alpha, 0)}.\]
Since we want $p>1$, we require $\alpha < d$, ie $\beta>0$.
\end{proof}

\begin{corollary}\label{lemmaupperc}
Let $K(r, r', y, y')$ be a kernel on the cone $M$ satisfying 
\[|K(r, r', y, y')|\leq
\begin{cases}
r^{-\alpha}r'^{-\beta},\hspace{5mm}&r\leq r',\\
0,& r> r'.\\
\end{cases}\]
If $\alpha+\beta=d$, $\beta>0$, and $p$ satisfies
\begin{equation}
p<\frac{d}{\max(\alpha, 0)},
\end{equation}
then $K$ is bounded as an operator on $L^p(M; r^{d-1}drdh)$.
\end{corollary}
\begin{proof}
It follows from Lemma \ref{lemmaupper} and 
the fact that the cross section $Y$ has finite volume.
\end{proof}

\begin{lemma}\label{lemmalower}
Consider the kernel $K(r, r')$ defined by
\[K(r, r')=
\begin{cases}
0, &r\leq r',\\
r^{-\alpha}r^{-\beta},\hspace{5mm}& r> r'.\\
\end{cases}\]
If $\alpha+\beta=d$, $\alpha>0$, and $p$ satisfies
\begin{equation}
p>\frac{d}{\min(\alpha, d)},
\end{equation}
then $K$ is bounded as an operator on $L^p(\mathbb{R}_+; r^{d-1}dr)$.
\end{lemma}
\begin{proof}
By duality and Lemma \ref{lemmaupper}.
\end{proof}

As before, Lemma \ref{lemmalower} has a corollary
about the boundedness of operators on the cone $M$. 

\begin{corollary}\label{lemmalowerc}
Let $K(r, r', y, y')$ be a kernel on the cone $M$ satisfying
\[|K(r, r', y, y')|\leq
\begin{cases}
0, &r\leq r',\\
r^{-\alpha}r^{-\beta},\hspace{5mm}& r> r'.\\
\end{cases}\]
If $\alpha+\beta=d$, $\alpha>0$, and $p$ satisfies
\begin{equation}
p>\frac{d}{\min(\alpha, d)},
\end{equation}
then $K$ is bounded as an operator on $L^p(M; r^{d-1}drdh)$.
\end{corollary}

\begin{proposition}\label{T2}
The operator $T_2$ is bounded on $L^p(M)$ for 
\begin{equation}\label{T2range}
p<\frac{d}{\max(\frac{d}{2}-\mu_0, 0)},
\end{equation}
where $\mu_0>0$ is the square root of the smallest eigenvalue of the operator $\Delta_Y+V_0(y)+(\frac{d-2}{2})^2$.
\end{proposition}
\begin{proof}
We will use the following three boundary defining functions
\[\rho_{\zf}=r',\hspace{2mm}\rho_{\lbz}=\frac{r}{r'},\hspace{2mm}\rho_{\rbi}=\langle r'\rangle^{-1}.\]
 By Theorem \ref{Gproperties}, we have
\begin{equation}\label{G2bound}
|G_2(r, r', y, y')|\lesssim\rho_{\zf}^{2-d}\rho_{\lbz}^{1-\frac{d}{2}+\mu_0}\rho_{\rbi}^\infty=r^{1-\frac{d}{2}+\mu_0}r'^{1-\frac{d}{2}-\mu_0}\langle r'\rangle^{-\infty},
\end{equation}
where $\mu_0>0$ is the square root of the smallest eigenvalue of the operator $\Delta_Y+V_0(y)+(\frac{d-2}{2})^2$.
From here we know that 
\[|G_2(\lambda r, \lambda r', y, y')|\lesssim
\begin{cases}
\lambda^{2-d}r^{1-\frac{d}{2}+\mu_0}r'^{1-\frac{d}{2}-\mu_0},\hspace{5mm}&\lambda\leq\frac{1}{r'},\\
\lambda^{1-\frac{d}{2}+\mu_0-N}r^{1-\frac{d}{2}+\mu_0}r'^{-N},&\lambda\geq\frac{1}{r'},
\end{cases}\]
for all $N>0$. That means, by polyhomogeneous conormality of $G_2$, 
if $\mu_0\ne\frac{d}{2}-1$, we have
\begin{equation}\label{gradient}
\big|\nabla_{z}\big(G_2(\lambda r, \lambda r', y, y')\big)\big|\lesssim
\begin{cases}
\lambda^{2-d}r^{-\frac{d}{2}+\mu_0}r'^{1-\frac{d}{2}-\mu_0},\hspace{5mm}&\lambda\leq\frac{1}{r'},\\
\lambda^{1-\frac{d}{2}+\mu_0-N}r^{-\frac{d}{2}+\mu_0}r'^{-N},&\lambda\geq\frac{1}{r'},
\end{cases}
\end{equation}
for all $N>0$. 

Using \eqref{Tidefn} and the fact that $G_2$ is supported where $r \leq r'$, we  estimate 
\begin{equation*}
\begin{split}
|T_2(r, r', y, y')|&\lesssim\int_0^\infty\lambda^{d-2}\big|\nabla_{z}\big(G_2(\lambda r, \lambda r', y, y')\big)\big|d\lambda\\
&\lesssim\int_0^\frac{1}{r'}\lambda^{d-2}\big(\lambda^{2-d}r^{-\frac{d}{2}+\mu_0}r'^{1-\frac{d}{2}-\mu_0}\big)d\lambda+
\int_{\frac{1}{r'}}^{\frac{1}{r}}\lambda^{d-2}\big(\lambda^{1-\frac{d}{2}+\mu_0-N}r^{-\frac{d}{2}+\mu_0}r'^{-N}\big)d\lambda\\
&=r^{-\frac{d}{2}+\mu_0}r'^{1-\frac{d}{2}-\mu_0}\int_0^\frac{1}{r'}d\lambda+
r^{-\frac{d}{2}+\mu_0}r'^{-N}\int_{\frac{1}{r'}}^{\frac{1}{r}}\lambda^{\frac{d}{2}+\mu_0-N-1}d\lambda\\
&=r^{-\frac{d}{2}+\mu_0}r'^{-\frac{d}{2}-\mu_0}+\frac{1}{\frac{d}{2}+\mu_0-N}\big(r^{N-d}r'^{-N}-r^{-\frac{d}{2}+\mu_0}r'^{-\frac{d}{2}-\mu_0}\big)\\
&\lesssim r^{-\frac{d}{2}+\mu_0}r'^{-\frac{d}{2}-\mu_0} \quad \text{ for } N >\mu_0+\frac{d}{2} \\
&=\Big(\frac{r}{r'}\Big)^{\mu_0-\frac{d}{2}}r'^{-d}.\\
\end{split}
\end{equation*}
By Corollary \ref{lemmaupperc}, we conclude that $T_2$ is bounded on $L^p(M)$ provided that
\[p<\frac{d}{\max(\frac{d}{2}-\mu_0, 0)}.\]
\end{proof}

\begin{remark}
When $V \equiv 0$, then $\mu_0=\frac{d}{2}-1$,
 and its first eigenfunction $u_0$ is a constant function. 
In Section~\ref{sec:proofs}
 we will improve estimate (\ref{gradient}) to obtain a bigger range for $p$ for this special case.
\end{remark}

\begin{proposition}\label{T3}
The operator $T_3$ is bounded on $L^p(M)$ for 
\[p>\frac{d}{\min(1+\frac{d}{2}+\mu_0, d)},\]
where $\mu_0>0$ is the square root of the smallest eigenvalue of the operator $\Delta_Y+V_0(y)+(\frac{d-2}{2})^2$.
\end{proposition}

\begin{proof}
We will use the following three boundary defining functions
\[\rho_{\zf}=r,\hspace{2mm}\rho_{\rbz}=\frac{r'}{r},\hspace{2mm}\rho_{\lbi}=\langle r\rangle^{-1}.\]
By Theorem \ref{Gproperties}, we have
\begin{equation}\label{G3bound}
|G_3(r, r', y, y')|\lesssim\rho_{\zf}^{2-d}\rho_{\rbz}^{1-\frac{d}{2}+\mu_0}\rho_{\lbi}^\infty=r^{1-\frac{d}{2}-\mu_0}r'^{1-\frac{d}{2}+\mu_0}\langle r\rangle^{-\infty},
\end{equation}
where $\mu_0>0$ is the square root of the smallest eigenvalue of the operator $\Delta_Y+V_0(y)+(\frac{d-2}{2})^2$. 
It follows that, as in the proof of Proposition \ref{T2}, by the polyhomogeneous conormality of $G_3$,
\[\big|\nabla_{z}\big(G_3(\lambda r, \lambda r', y, y')\big)\big|\lesssim
\begin{cases}
\lambda^{2-d}r^{-\frac{d}{2}-\mu_0}r'^{1-\frac{d}{2}+\mu_0},\hspace{5mm}&\lambda\leq\frac{1}{r},\\
\lambda^{-\frac{d}{2}+\mu_0-N+1}r^{-N-1}r'^{1-\frac{d}{2}+\mu_0},&\lambda\geq\frac{1}{r},
\end{cases}\]
for all $N>0$. Then using \eqref{Tidefn} and the fact that $G_3$ is supported where $r' \leq r$, we have 
\begin{equation*}
\begin{split}
|T_3(r, r', &y, y')|\lesssim\int_0^\infty\lambda^{d-2}\big|\nabla_{z}\big(G_3(\lambda r, \lambda r', y, y')\big)\big|d\lambda\\
&\lesssim \int_0^\frac{1}{r}\lambda^{d-2}\big(\lambda^{2-d}r^{-\frac{d}{2}-\mu_0}r'^{1-\frac{d}{2}+\mu_0}\big)d\lambda+
\int_{\frac{1}{r}}^{\frac{1}{r'}}\lambda^{d-2}\big(\lambda^{-\frac{d}{2}+\mu_0-N+1}r^{-N-1}r'^{1-\frac{d}{2}+\mu_0}\big) d\lambda\\
&=r^{-\frac{d}{2}-\mu_0}r'^{1-\frac{d}{2}+\mu_0}\int_0^\frac{1}{r}d\lambda+
r^{-N-1}r'^{1-\frac{d}{2}+\mu_0}\int_{\frac{1}{r}}^{\frac{1}{r'}}\lambda^{\frac{d}{2}+\mu_0-N-1}d\lambda\\
&=r^{-1-\frac{d}{2}-\mu_0}r'^{1-\frac{d}{2}+\mu_0}+\frac{1}{\frac{d}{2}+\mu_0-N}\big(r^{-N-1}r'^{N-d+1}-r^{-1-\frac{d}{2}-\mu_0}r'^{1-\frac{d}{2}+\mu_0}\big)\\
&\lesssim r^{-1-\frac{d}{2}-\mu_0}r'^{1-\frac{d}{2}+\mu_0} \quad \text{ for } N >\mu_0+\frac{d}{2}\\
&=\Big(\frac{r'}{r}\Big)^{\mu_0-\frac{d}{2}+1}r^{-d}.\\
\end{split}
\end{equation*}
Applying Corollary \ref{lemmalowerc}, we conclude that $T_3$ is bounded on $L^p(M)$ provided that
\[p>\frac{d}{\min(1+\frac{d}{2}+\mu_0, d)}.\]
\end{proof}

\begin{proposition}\label{T2T3}
The operator $T_1$ is bounded on $L^2(M)$.
\end{proposition}
\begin{proof}
Since $2$ satisfies the boundedness criteria in both Proposition \ref{T2} and Proposition \ref{T3}, 
the operator $T_2+T_3=T-T_1$ is bounded on $L^2(M)$. 
The operator $T$ is bounded on $L^2(M)$ by Proposition \ref{bon2}, 
and from here the boundedness of $T_1$ on $L^2(M)$ follows.
\end{proof}

\begin{remark}
Proposition \ref{T2T3} completes the missing part in the proof of Proposition \ref{T1}.
\end{remark}

\subsection{Proofs of main results}\label{sec:proofs}

\begin{proof}[Proof of Theorem~\ref{fmain}]
Since $T=T_1+T_2+T_3$, we just combine Proposition \ref{interpolation}, Proposition \ref{T2} and Proposition \ref{T3} 
to prove the first part of this theorem.

For the second part, with $V\not\equiv 0$, for $p$ outside the interval (\ref{ccinterval}), the counterexamples from \cite[Section 5.2]{GH} serve to show the lack of boundedness of $T$ on $L^p(M)$. (For purposes of comparison, note that 
the variables $x$ and $x'$ in \cite{GH} 
correspond to $\frac{1}{r}$ and $\frac{1}{r'}$ in this paper.)
\end{proof}

\begin{proof}[Proof of Theorem~\ref{HQLR}] 
Suppose that the potential $V$ is identically zero; we proceed to show that the upper threshold for $L^p$ boundedness is $p = d (d/2 - \mu_1)^{-1}$. Notice that $T_1$ and $T_3$ are automatically bounded on this extra range, so we only have to consider $T_2$, which has an expression of the form  
\begin{equation}
T_2(z,z') = \frac{2}{\pi} \chi(4r/r') \int_0^\infty\lambda^{d-2}\nabla_z\big(G(\lambda z, \lambda z')\big) d\lambda.\
\label{T2int}\end{equation}
We recall that $G = r r'\tilde G$ and substitute the infinite series \eqref{infiniteseries} for $\tilde G$ here, and consider the first term in this sum separately from the rest.  Since $\mu_0 = d/2 - 1$ when $V_0 = 0$, the first term here is (as a multiple of the Riemannian half-density --- recall this gives us an extra factor of $(r r')^{-d/2}$, as in \eqref{KKtilde}) 
$$
(r r')^{1-d/2} u_0(y) u_0(y') I_{d/2 - 1}(r) K_{d/2 - 1}(r')
$$
When $V_0 = 0$, the eigenfunction $u_0(y)$ is constant. Also, $I_{d/2 - 1}(r) = c r^{d/2 - 1} + O(r^{d/2 + 1})$ and is conormal at $r=0$, implying that $\nabla_r (r^{1-d/2} I_{d/2 - 1}(r)) = O(r)$. Hence $$\nabla_z ( r^{1-d/2} u_0(y) I_{d/2 - 1}(r) ) = O(r);$$ that is, in this special case, applying the derivative $\nabla_z$ makes the kernel vanish to an additional order, instead of one order less as is usually the case. 
Therefore, after taking the gradient in the left variables, this term is bounded by 
$$
\begin{cases} C r {r'}^{2-d}, \quad r' \leq 1 \\
C r {r'}^{-N}  , \quad r' \geq 1 
\end{cases}
$$
for any integer $N$. 
Now we put this in \eqref{T2int} and find that the contribution to $T_2$ of the $\mu_0$-term is bounded by 
$$\begin{gathered}
\int_0^{1/r'} \lambda^{2-d} \Big( \lambda^{d} r {r'}^{2-d} \Big) \, d\lambda + \int_{1/r'}^\infty \lambda^{2-d} \Big( \lambda^{2-N} r {r'}^{-N} \Big) \, d\lambda \\
\leq C r {r'}^{-1-d}. 
\end{gathered}$$
Remembering that this term is supported in $\{ r \leq r' \}$, we see from Corollary~\ref{lemmaupperc} that this term is bounded on $L^p$ for all $p \in (1, \infty)$. 

So consider the remainder of the series. The argument in the previous subsection applies, except that the series now begins with the $\mu_1$ term rather than the $\mu_0$ term, so we have boundedness in the range \eqref{T2range} with $\mu_1$ replacing $\mu_0$, completing the proof. 
\end{proof}

\end{document}